\documentclass[a4paper,reqno,10pt]{article}

\usepackage{amsmath}
\usepackage{amssymb,mathrsfs}
\usepackage{bm}
\IfFileExists{bbm.sty}{\usepackage{bbm}}{} 
\usepackage{graphicx}
\usepackage{xcolor}
\usepackage{amsfonts}
\usepackage{amsthm}
\usepackage{enumerate,enumitem}
\newtheorem{definition}{Definition}[section]
\newtheorem{theorem}[definition]{Theorem}
\newtheorem{proposition}[definition]{Proposition}
\newtheorem{remark}[definition]{Remark}

\newtheorem{lemma}[definition]{Lemma}
\newtheorem{example}[definition]{Example}
\newtheorem{problem}[definition]{Problem}

\usepackage{indentfirst}
\usepackage{multirow}
\numberwithin{equation}{section}

\usepackage[margin=1in]{geometry}
\usepackage[english]{babel}
\usepackage[T1]{fontenc}
\usepackage{times}

\date{}
\title{\bf Boundary Control of Two-dimensional Point-Vortex Systems}
\author{Ka\"{i}s Ammari\footnote{Universit\'e de Monastir, Facult\'e des Sciences de Monastir, D\'epartement de Math\'ematiques, 5019 Monastir, Tunisia. \textit{Email:} \texttt{kais.ammari@fsm.rnu.tn}.}
\quad Taoufik Hmidi\footnote{New York University Abu Dhabi, Abu Dhabi, United Arab Emirates. \textit{Email:} \texttt{th2644@nyu.edu}.}
\quad Haocheng Yang\footnote{New York University Abu Dhabi, Abu Dhabi, United Arab Emirates. \textit{Email:} \texttt{hy3448@nyu.edu}.}}

\usepackage[citestyle=numeric,maxbibnames=9, maxcitenames=9,doi=false,isbn=false,url=false,giveninits=true,eprint=false,bibstyle=numeric]{biblatex}
\renewbibmacro{in:}{}
\AtEveryBibitem{%
\clearlist{language}%
}
\usepackage{csquotes}

\usepackage{tikz}
\tikzset{every picture/.style={line width=0.75pt}} 
\usepackage{caption}
\usepackage[colorlinks,linkcolor=blue,citecolor=red,urlcolor=black]{hyperref}

\usepackage{verbatim} 

\begin{document}

\pagenumbering{arabic}

\newcommand{\Op}[2][]{\operatorname{Op}^{#1}\left(#2\right)}
\newcommand{\Supp}[1]{\operatorname{Supp} #1}
\newcommand{\RE}{\operatorname{Re}}
\newcommand{\IM}{\operatorname{Im}}
\newcommand{\Id}{\mathrm{Id}}
\newcommand{\diver}{\operatorname{div}}
\newcommand{\curl}{\operatorname{curl}}
\newcommand{\tr}[1]{\operatorname{tr}\left(#1\right)}
\newcommand{\di}{\operatorname{d} \!}
\newcommand{\dist}{\operatorname{dist}}

\newcommand{\xR}{\mathbb{R}}
\newcommand{\xC}{\mathbb{C}}
\newcommand{\xN}{\mathbb{N}}
\newcommand{\xT}{\mathbb{T}}
\newcommand{\xZ}{\mathbb{Z}}
\newcommand{\xD}{\mathbb{D}}
\newcommand{\xS}{\mathbb{S}}

\newcommand{\M}{\mathcal{M}}

\renewcommand{\le}{\leqslant}
\renewcommand{\ge}{\geqslant}
\renewcommand{\leq}{\leqslant}
\renewcommand{\geq}{\geqslant}

\newcommand{\GL}{\mathrm{GL}}
\newcommand{\PV}{\mathrm{PV}}
\newcommand{\cor}{\mathrm{cor}}
\newcommand{\main}{\mathrm{main}}
\newcommand{\bdy}{\mathrm{bdy}}
\newcommand{\Dir}{\mathrm{D}}
\newcommand{\Neu}{\mathrm{N}}
\newcommand{\Tr}{\mathrm{Tr}}
\newcommand{\ext}{\mathrm{ext}}
\newcommand{\per}{\mathrm{per}}
\newcommand{\loc}{\mathrm{loc}}
\newcommand{\refer}{\mathrm{ref}}
\newcommand{\initial}{\mathrm{in}}
\newcommand{\final}{\mathrm{en}}
\newcommand{\app}{\mathrm{app}}

\numberwithin{equation}{section}
\pagestyle{plain}

\def\la{\left\lvert}
\def\lA{\left\lVert}
\def\ra{\right\rvert}
\def\rA{\right\rVert}

\maketitle

\begin{abstract}
    In this article, we investigate the boundary controllability of finite point-vortex systems in bounded multiply connected planar Lipschitz domains. The controls are applied locally on an arbitrary nonempty relatively open subset of the boundary. Zero incoming vorticity is assumed at the inflow part of the boundary. For time-dependent controls, we prove exact trajectory controllability along arbitrary prescribed collision-free vortex paths remaining away from the boundary. For stationary controls, we prove exact endpoint controllability associated with mutually disjoint real-analytic arcs. The proofs rely on harmonic interpolation, localized Runge approximation, and an implicit-function argument. The optimal control minimizing the length of trajectories is also studied. We prove the nonexistence of minimizers within classes of slightly smooth controls, and establish existence when the control time and the size of the control are penalized.
\end{abstract}

\medskip



\tableofcontents

\section{Introduction}\label{sect:Intro}

The dynamics of point vortices in two-dimensional incompressible inviscid fluids constitute one of the classical finite-dimensional models of fluid mechanics, dating back to the pioneering works of Helmholtz~\cite{Helmholtz1858} and Kirchhoff \cite{Kirchhoff1876}. In this description, the vorticity is concentrated at finitely many moving points,
$$
    \omega(t)=\sum_{j=1}^N \kappa_j\delta_{X_j(t)},
$$
where $\kappa_j\neq0$ is the circulation carried by the $j$-th vortex and $X_j(t)$ denotes its position. After removing the singular self-interaction, the Euler dynamics reduce to a finite-dimensional Hamiltonian system for the vortex positions. In a bounded domain, this system is determined by the Green function of the Laplacian and its regular part, together with additional circulation terms when the domain is multiply connected. The point-vortex model therefore retains several essential geometric features of the Euler dynamics while replacing the evolution of an infinite-dimensional vorticity field by a system of ordinary differential equations. It has consequently played an important role both in the mathematical study of vortex dynamics and as a reduced-order model for fluid motion; see, for instance,~\cite{Newton2001}.

\medskip

The purpose of this paper is to study the controllability of this system
when the control is applied exclusively on the boundary of the
fluid domain. More precisely, we ask whether a finite collection of point
vortices inside a bounded planar domain can be steered between prescribed
configurations, or even forced to follow prescribed trajectories, by
acting only on an arbitrarily small nonempty relatively open portion of
the boundary.
The control mechanism considered here can be viewed as control on harmonic background flow, taking the form 
$$
    u_{\rm c}=\nabla\Phi, \qquad \Delta\Phi=0\quad\hbox{in }\Omega,
$$
where the harmonic potential $\Phi$ is generated by boundary data supported in a prescribed relatively open subset
$\Sigma\subset\partial\Omega$. Both Dirichlet and Neumann conditions are considered. The basic mechanism is illustrated in Figure~\ref{fig:vortex}.

\medskip

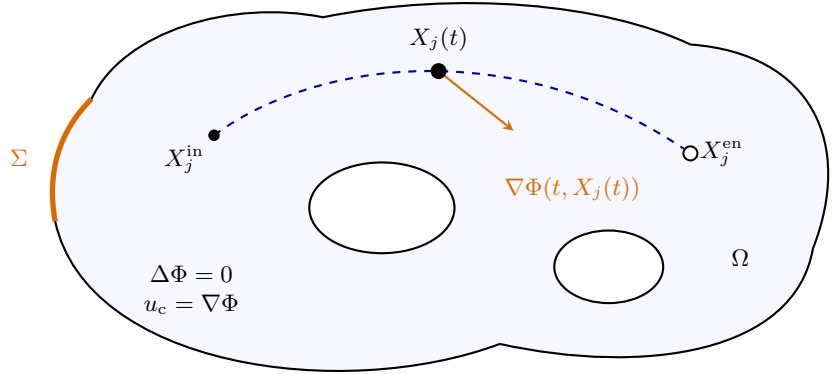
\begin{figure}[htbp]
\centering
\begin{tikzpicture}[scale=1.2, font=\small, >=stealth]
    \draw[thick, fill=blue!3] (-4.25,-0.3)
        .. controls (-4.35,0.25) and (-4.2,0.7) .. (-3.85,1.05)
        .. controls (-3.5,1.8) and (-2.3,2.15) .. (-1.3,1.95)
        .. controls (0.2,2.25) and (1.8,2.1) .. (2.75,1.65)
        .. controls (4.1,1.55) and (4.55,0.55) .. (4.1,-0.6)
        .. controls (3.9,-1.75) and (2.25,-2.0) .. (0.65,-1.65)
        .. controls (-1.1,-2.3) and (-3.9,-1.9) .. (-4.25,-0.3);
    \draw[thick, fill=white] (-0.65,-0.15) ellipse (0.8 and 0.5);
    \draw[thick, fill=white] (1.85,-0.8) ellipse (0.6 and 0.4);
    \node at (3.3,-0.7) {$\Omega$};

    \draw[orange!85!black, line width=2pt] (-4.25,-0.3)
        .. controls (-4.35,0.25) and (-4.2,0.7) .. (-3.85,1.05);
    \node[orange!85!black] at (-4.65,0.4) {$\Sigma$};
    \node[align=center] at (-2.75,-1.05)
        {$\Delta\Phi=0$\\$u_{\rm c}=\nabla\Phi$};

    \draw[blue!65!black, dashed, thick] (-2.5,0.65)
        .. controls (-1.2,1.6) and (1.05,1.65) .. (2.75,0.45);
    \fill (-2.5,0.65) circle (1.8pt)
        node[below left] {$X_j^\initial$};
    \draw[fill=white] (2.75,0.45) circle (2.2pt)
        node[right] {$X_j^\final$};
    \coordinate (vortex) at (-0.025,1.35625);
    \draw[->, orange!85!black, thick] (vortex) -- (0.8,0.7);
    \node[orange!85!black, below] at (1.45,0.3)
        {$\nabla\Phi(t,X_j(t))$};
    \fill (vortex) circle (2.5pt)
        node[above=4pt] {$X_j(t)$};
\end{tikzpicture}
\caption{Schematic representation of localized boundary control of
point vortices. Boundary data supported on
$\Sigma\subset\partial\Omega$ generate a harmonic velocity field in
$\Omega$, which modifies the vortex dynamics through its value at the
vortex position.}
\label{fig:vortex}
\end{figure}

Here, we need to clarify the relation between the finite-dimensional model considered here and the Euler equation with inflow and outflow. Since a Neumann boundary control may have a nonzero normal component, prescribing the normal velocity alone is not sufficient to determine a solution of the full vorticity equation: the vorticity entering the domain through the inflow boundary must also be specified. For the full Euler equation, the well-posedness is only proved for source-sink case~\cite{sueur2022existence,
gie2025euler,sueur2023uniqueness}, where the vorticity is allowed to flow either inward or outward on each component of the boundary; or the analytical case~\cite{kukavica2025inflow-outflow}, where the uniqueness comes from the analytic assumption. In the present paper, we remain throughout in the point-vortex class and impose zero incoming vorticity. Whenever the normal trace $g=u\cdot n$ admits a pointwise representative, this corresponds to
\[
       \omega_{\rm in}=0
       \qquad\hbox{on }\{g<0\}.
\]
Thus, the controlled dynamics form a self-contained ordinary differential equation on the collision-free configuration space. 

\medskip

The singularity of this self-contained system occurs only when two vortices approach one another or when a vortex approaches the boundary. The reference trajectories in our controllability results are therefore required to remain uniformly separated from collisions and from $\partial\Omega$. For comparison, in the uncontrolled point-vortex system, global existence for almost every initial configuration is known under stronger regularity assumptions on the boundary; see~\cite{donati2022point}.

\medskip

This problem is related to, but distinct from, several classical directions
in the control of incompressible fluids. Boundary controllability of the
Euler equations was initiated by Coron
\cite{coron1996controllability}, who established fundamental
controllability results for two-dimensional incompressible perfect
fluids. In this theory, the topology of the domain plays an important
role. When the fluid region is multiply connected, Kelvin circulation
invariants impose constraints on the reachable Eulerian states. Glass
subsequently refined the planar theory in multiply connected domains
\cite{glass2001addendum} and established exact boundary controllability
for the three-dimensional Euler equation \cite{glass2000exact}. These
works concern controllability at the level of the Eulerian fluid state,
rather than the direct steering of a finite system of singular vortices.
A related but conceptually different problem is Lagrangian
controllability, where one seeks to control trajectories of fluid
particles or the motion of material sets without prescribing the complete
Eulerian state. Glass and Horsin
\cite{glassHorsin2010lagrangian} established approximate Lagrangian
controllability for the two-dimensional Euler equation and applied it to
the approximate control of the shape of vortex patches. This point of
view is closer in spirit to the present paper because the object being
controlled is a distinguished geometric component of the flow rather
than the entire velocity field. There is nevertheless an important
difference: a point vortex is not a passive fluid particle. Its velocity
contains the interactions with all the other vortices, as well as the
regular self-interaction induced by the boundary. Controlling point
vortices therefore amounts to steering a nonlinear Hamiltonian system
associated with singular vorticity.

\medskip

Control problems formulated directly at the level of finite point-vortex
systems have also attracted attention. Dorsz and Glass
\cite{DG2023control} proved exact controllability in arbitrary positive
time by prescribing the motion of one additional point vortex, which acts
as a control for the remaining vortices. In a different direction,
motivated in particular by reduced-order flow control, Nelson, Protas and
Sakajo \cite{nelsonProtasSakajo2017} studied controllability and
observability of the linearized dynamics near point-vortex equilibria and
constructed feedback controls based on localized blowing and suction.
More generally, we refer to \cite{vainchteinMezic2006} for an overview of
vortex-based control algorithms and related reduced-order approaches.

\medskip

The control mechanism considered in the present paper differs from
these approaches. No additional vortex is introduced as a control, and
the vortices are not controlled directly. Instead, the control is the
harmonic velocity field generated by Dirichlet or Neumann data supported
in a prescribed boundary patch. The control is therefore localized at
the physical boundary but acts nonlocally in the interior through
harmonic propagation. Moreover, our principal results concern exact
nonlinear steering: time-dependent boundary controls yield exact
trajectory controllability, whereas stationary controls yield exact
endpoint controllability near prescribed analytic arcs.

\medskip

Throughout the paper, the fluid domain satisfies
\begin{equation}\tag{H}\label{hyp:Domain}
    \begin{gathered}
        \Omega\subset\xR^2 \text{ is a bounded connected Lipschitz domain, whose boundary has}\\
        \text{$(m+1)$ connected components } \partial\Omega=\Sigma_0\sqcup\Sigma_1\sqcup\cdots\sqcup\Sigma_m,
        \quad (m\in\xN)
    \end{gathered}
\end{equation}
where $\Sigma_0$ is the exterior boundary component and
$\Sigma_1,\ldots,\Sigma_m$ surround the holes. When $m=0$, the domain is
simply connected. 
\begin{figure}[htbp]
\centering
\begin{tikzpicture}[scale=1, font=\small]
    \draw[thick, fill=blue!3] (-4.2,-0.15)
        .. controls (-4.4,0.9) and (-3.4,1.9) .. (-1.2,1.7)
        -- (0,1.85) -- (0.8,0.95) -- (1.7,1.75)
        .. controls (3.2,2.0) and (4.4,1.1) .. (4.2,-0.25)
        .. controls (4.1,-1.75) and (2.15,-1.95) .. (0.65,-1.65)
        .. controls (-1.1,-2.3) and (-3.9,-1.9) .. (-4.2,-0.15);
    \draw[thick, fill=white] (-1.7,-0.15) ellipse (0.8 and 0.55);
    \draw[thick, fill=white] (2.1,-0.4) ellipse (0.65 and 0.6);
    \node at (-2.9,1.9) {$\Sigma_0$};
    \node at (-1.7,0.7) {$\Sigma_1$};
    \node at (2.1,0.5) {$\Sigma_2$};
    \node at (0,-0.95) {$\Omega$};

    \begin{scope}[shift={(-1.7,-0.15)}]
        \draw[orange!85!black, line width=2pt]
            ({0.8*cos(210)},{0.55*sin(210)})
            arc[start angle=210,end angle=285,x radius=0.8,y radius=0.55];
    \end{scope}
    \node[orange!85!black] at (-2.05,-1.0) {$\Sigma$};
\end{tikzpicture}
\caption{Geometry of hypothesis~\eqref{hyp:Domain}, illustrated for $m=2$.
The fluid region $\Omega$ is bounded by the exterior component $\Sigma_0$
and the hole boundaries $\Sigma_1,\Sigma_2$. The control set $\Sigma$ may
have nonempty intersection with any boundary component.}
\label{fig:domain}
\end{figure}
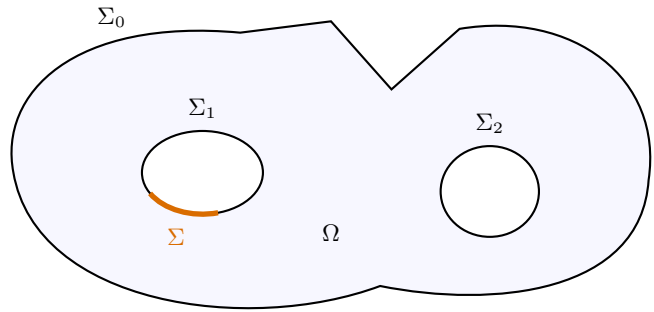

In a multiply connected domain, prescribing the
vorticity and the normal component of the velocity does not uniquely
determine the velocity field: one must in addition prescribe the
circulations around the holes. These degrees of freedom are represented
by a finite-dimensional space of divergence-free and curl-free harmonic
vector fields. We therefore fix once and for all the circulations
\[
       \mathcal C_\ell
       =\int_{\Sigma_\ell}u\cdot\tau\,d\sigma,
       \qquad \ell=1,\ldots,m.
\]
Once this circulation class is fixed, its contribution to the
point-vortex equations is a smooth position-dependent vector field on the
collision-free configuration space and can be incorporated into the
uncontrolled drift. Thus, the topology affects the circulation
constraints and the existence of global harmonic conjugates, but it does
not obstruct the localized harmonic interpolation and approximation
mechanisms used in our controllability arguments. In the simply connected
case $m=0$, these circulation terms are absent.

\medskip

The first main contribution of this paper is an exact trajectory
controllability result for time-dependent boundary controls. Let
$\Sigma\subset\partial\Omega$ be an arbitrary nonempty relatively open
set. Given prescribed smooth trajectories of the $N$ vortices that remain
mutually distinct and at positive distance from the boundary, we construct
boundary data supported in $\Sigma$ for which these trajectories are exact
solutions of the controlled point-vortex system. The result holds for both
Dirichlet controls and compatible Neumann controls. Its main elliptic
ingredient is a finite-point harmonic interpolation result, based on
boundary unique continuation, which allows one to prescribe the gradient
of a harmonic potential simultaneously at the vortex positions.

\medskip

The second main contribution concerns stationary boundary controls, for
which the same harmonic potential must govern the motion throughout the
whole time interval. Given mutually disjoint real-analytic arcs joining
the initial and final vortex positions, we prove exact endpoint
controllability by a stationary boundary control, with the corresponding
trajectories converging, after re-parametrization, to the prescribed arcs
in every $C^k$ norm. The proof combines local harmonic realization near
the arcs, localized Runge approximation, and an implicit-function argument that corrects the remaining endpoint error.
These two results rely on different mechanisms. Time-dependent control
allows the vortex velocities to be prescribed instantaneously, whereas
stationary control requires a global approximation-and-correction
procedure. In both cases, the control region may be an arbitrarily small
nonempty relatively open subset of the boundary. In particular, for the
point-vortex system with the circulation class fixed, the control region
need not intersect every connected component of $\partial\Omega$.

\medskip

Beyond controllability, we investigate the geometry and optimization of
controlled vortex trajectories. For a single vortex, we consider the
infimum of the lengths of trajectories generated by stationary localized
boundary controls and prove that it coincides with the intrinsic distance
of the domain. Thus, despite the rigidity imposed by harmonic and
stationary control, controlled trajectories can have lengths
arbitrarily close to the shortest geometric length permitted by the
domain. In convex domains with real-analytic boundary, a further rigidity phenomenon occurs. Although the intrinsic distance between two points coincides with their Euclidean distance, the corresponding minimizing trajectory---the straight segment joining the endpoints---cannot be realized by any regular stationary Dirichlet or Neumann control. Thus, the optimal length can be approached arbitrarily closely by admissible controlled trajectories, but it is not attained unless the control is singular enough. The underlying obstruction is Hamiltonian. An exactly length-minimizing
trajectory would have to be the straight segment joining its endpoints,
whereas every stationary one-vortex trajectory lies on a level set of the
controlled Kirchhoff--Routh Hamiltonian. The boundary behavior of the
Robin function rules out such a level set containing a nontrivial
segment. We also introduce a regularized
optimal-control problem and prove the existence of minimizers when the control time and the size of control function are penalized.

\subsection{Formulation of the system}
Let $\Omega$ be as above (see the hypothesis~\eqref{hyp:Domain}). We consider the motion of an incompressible perfect fluid in $\Omega$:
\begin{equation}\label{eq:Euler}
    \left\{\begin{aligned}
        &\partial_t u + (u \cdot \nabla) u + \nabla p = 0, \quad &&\text{in } \Omega,\\
        &\diver u = 0, \quad &&\text{in } \Omega.
    \end{aligned}\right.
\end{equation}
In 2D, the vorticity $\omega = \curl u = \partial_x u^y - \partial_y u^x$ (where $u=(u^x,u^y)$) is governed by the active scalar equation
\begin{equation}\label{eq:Euler_Vor}
    \partial_t \omega + (u\cdot\nabla)\omega = 0,
\end{equation}
where the velocity field $u$ can be recovered through the Biot-Savart law and the Helmholtz decomposition,
\begin{equation}\label{eq:BSLaw}
    u(t,X) = \int_{\Omega} \nabla_X^\perp G_{\Omega}(X,Y) \omega(t,Y) \di Y + \sum_{r=1}^m c_r(t)\nabla^\perp W_r(X) + \nabla \Phi(t,X).
\end{equation}
Here $G_\Omega(X,Y)$ is the Dirichlet Green function of the Laplacian in $\Omega$, namely,
$$
    \Delta_XG_\Omega(X,Y)=\delta_Y(X),\qquad
    G_\Omega(\cdot,Y)|_{\partial\Omega}=0,
$$
and $\Phi$ is a single-valued harmonic velocity potential determined by the boundary control. $W_r$'s ($r=0,1,...,m$) are harmonic measures determined by
\begin{equation}\label{eq:HarMes}
    \Delta W_r = 0, \text{ in $\Omega$}; \quad W_r|_{\Sigma_r} =1, \quad W_r|_{\partial\Omega\backslash\Sigma_r} =0,
\end{equation}
and $c_r(t)$ are real coefficients depending on time. We also used the notation
$$
    X=(x,y),\qquad
    \nabla^\perp = \nabla_X^\perp=(-\partial_y,\partial_x)^T.
$$
We denote by $n$ the almost everywhere defined outward unit normal, set $\tau=n^\perp$, and write $\di\sigma$ for the boundary measure. Then we prescribe the circulations around each hole,
$$
    \mathscr C_l:=\int_{\Sigma_l}u\cdot\tau\di\sigma, \qquad l=1,...,m,
$$
where the integral is understood in the weak tangential-trace sense, and set
$$
    \mathsf A_{lr}:=\left\langle n\cdot\nabla W_r,W_l|_{\partial\Omega}\right\rangle_{H^{-1/2}(\partial\Omega),H^{1/2}(\partial\Omega)} = \int_{\Sigma_l}n\cdot\nabla W_r\di\sigma, \qquad 1\leq l,r\leq m.
$$
Note that the last equality holds whenever the conormal derivative is integrable. Since $\mathsf A$ is symmetric positive definite, and hence invertible,
the coefficients $c_r(t)$ are uniquely determined by the vorticity
$\omega$, the harmonic measures $W_r$, and the prescribed circulations
$\mathscr C_r$. More precisely,
\[
    (c_r(t))_{r=1}^m
    =
    \mathsf A^{-1}
    \left(
        \mathscr C_l
        -\int_\Omega W_l(Y)\omega(t,Y)\di Y
    \right)_{l=1}^m.
\]
Indeed, this identity follows by taking the circulation of
\eqref{eq:BSLaw} along each inner boundary component $\Sigma_l$,
observing that the gradient term has zero circulation, and applying
Green's identity to the Biot--Savart contribution. Thus, in a multiply connected domain $\Omega$, with circulations $\mathscr{C}_r$'s fixed, the vorticity $\omega$ and harmonic background flow $\nabla\Phi$ uniquely determine the velocity $u$. The simply connected case $m=0$ is recovered by omitting the sums above.

\medskip

Let us formulate the point vortex problem. Assume that the vorticity $\omega$ is a linear combination of $N$ point vortices for all time:
\begin{equation}\label{eq:vorticity}
    \omega(t,X) = \sum_{j=1}^N \kappa_j \delta_{X_j(t)}(X) = \sum_{j=1}^N \kappa_j \delta_0(X - X_j(t)),
\end{equation}
where the nonzero real number $\kappa_j \in \xR$ is the intensity of the $j$-th point vortex and $X_j(t)$ is its position at time $t$. To pass to the evolution of $X_j$'s, we recall the singular decomposition of Green function $G_\Omega(X,Y)$,
$$
    G_\Omega(X,Y) = G_{\xR^2}(X,Y) + \tilde{R}_{\Omega}(X,Y) = \tfrac{1}{2\pi} \ln |X - Y| + \tilde{R}_\Omega(X,Y),
$$
where $\tilde{R}_\Omega(X,Y)$ is a symmetric real-analytic function on $\Omega\times\Omega$, known as the regular part of the Green function $G_\Omega$. Set
$$
    R_\Omega(X):=\tilde{R}_\Omega(X,X).
$$
For the vorticity~\eqref{eq:vorticity}, with $\bm X=(X_1,...,X_N)$, we write
$$
    (c_r(\bm X))_{r=1}^m
    :=\mathsf A^{-1}\left(\mathscr C_l-\sum_{k=1}^N\kappa_kW_l(X_k)\right)_{l=1}^m, \qquad \bm X \in \mathcal{C}_N(\Omega),
$$
where $\mathcal{C}_N(\Omega)$ denotes the collision-free configuration space,
\begin{equation}\label{eq:ConfigSet}
    \mathcal{C}_N(\Omega) := \left\{ \bm X=(X_1,\ldots,X_N)\in\Omega^N: X_i\neq X_j\ \text{for }i\neq j \right\}.
\end{equation}
Then the motion of the point vortices is governed by the Kirchhoff-Routh equation:
\begin{equation}\label{eq:KR}
    \begin{aligned}
        \dot{X}_j(t) = &\sum_{k \neq j} \kappa_k \nabla_X^\perp G_\Omega(X_j(t), X_k(t))
        + \frac{\kappa_j}{2} \nabla^\perp R_\Omega(X_j(t))\\
        &\hspace{6em}+ \sum_{r=1}^m c_r(\bm X(t))\nabla^\perp W_r(X_j(t)) + \nabla\Phi(t,X_j(t)).
    \end{aligned}
\end{equation}
The point vortex system~\eqref{eq:KR} can be derived as a limit system of~\eqref{eq:Euler_Vor}-\eqref{eq:BSLaw} when the vorticity concentrates at several points; see the derivation in Marchioro's paper~\cite{marchioro1988euler}. One may also refer to the works of Glass, Lacave, Munnier, and Sueur~\cite{glass2016body,glass2018limit} for the derivation of~\eqref{eq:KR} as the limit of small solids in an incompressible perfect fluid.
\begin{remark}
    If the velocity field $u$ has nontrivial boundary condition, the vorticity form of Euler equation~\eqref{eq:Euler_Vor}-\eqref{eq:BSLaw} requires the information of incoming vorticity,
    $$
        \omega_{\rm in} := \omega\text{ restricted to }\{(t,Y)\in[0,+\infty[\times\partial\Omega:(n\cdot u)(t,Y)<0\}
    $$
    to obtain the uniqueness (see, for instance, the paper of Noisette and Sueur~\cite{sueur2023uniqueness}). This fact does not contradict the self-contained system~\eqref{eq:KR} above. Indeed, we have assumed a priori that the vorticity $\omega$ is composed of the $N$ point vortices above for all time. This suggests the following zero-inflow condition,
    \begin{equation}\label{eq:ZeroInflow}
        \omega_{\rm in}=0
    \end{equation}
    Thus, the rigorous derivation of~\eqref{eq:KR} should follow this condition.
\end{remark}

\begin{remark}\label{rmk:Hamilton}
    It is worth mentioning that the point vortex system without background flow $\nabla\Phi$ is Hamiltonian,
    $$
        \kappa_j\dot{X}_j = \nabla^\perp_{X_j} \mathcal{H}_0(\bm X),
    $$
    where the Hamiltonian reads
    $$
        \mathcal{H}_0(\bm X) = \frac{1}{2} \sum_{j\neq k} \kappa_j \kappa_k G_\Omega(X_j,X_k) + \frac{1}{2} \sum_{j=1}^N\kappa_j^2 R_\Omega(X_j) - \frac{1}{2} \bm b(\bm X)^T \mathsf{A}^{-1} \bm b(\bm X),
    $$
    with 
    $$
        \bm b (\bm X)  =\left(\mathscr C_l-\sum_{k=1}^N\kappa_kW_l(X_k)\right)_{l=1}^m.
    $$
    If the background flow is nonzero, this system is only locally Hamiltonian since the harmonic conjugate of $\Phi$ can only be defined locally. However, when the background potential $\Phi$ is time-independent and admits a single-valued harmonic conjugate $\Psi$ defined on $\Omega$, the system~\eqref{eq:KR} becomes Hamiltonian with Hamiltonian function
    $$
        \mathcal{H}(\bm X) := \mathcal{H}_0(\bm X) + \sum_{j=1}^N \kappa_j \Psi(X_j).
    $$
\end{remark}

\subsection{Boundary control}

In this paper, we are interested in the boundary controllability of the point vortex system~\eqref{eq:KR}. More precisely, we want to know whether it is possible to control the motion of the point vortices from an initial configuration to a final configuration through some control $f$ on Dirichlet boundary condition
\begin{equation}\label{eq:CtrlD}
    \left\{\begin{aligned}
        & \Delta \Phi = 0, \quad &&\text{in } \Omega,\\
        &\Phi|_{\partial\Omega} = f,
    \end{aligned}\right.
\end{equation}
or some control $g$ on Neumann boundary condition
\begin{equation}\label{eq:CtrlN}
    \left\{\begin{aligned}
        & \Delta \Phi = 0, \quad &&\text{in } \Omega,\\
        &n\cdot\nabla\Phi|_{\partial\Omega} = g,
    \end{aligned}\right.
\end{equation}
with a compatibility condition
\begin{equation}\label{eq:CtrlN_Comp}
    \int_{\partial\Omega} g \di \sigma = 0.
\end{equation}
Since the Biot-Savart and harmonic-circulation parts are tangent to the boundary, the function $g$ coincides with the normal velocity $n\cdot u|_{\partial\Omega}$. 

\begin{remark}\label{rmk:DtN-localization}
    At the natural trace level, these controls are equivalent up to constants. More precisely, the Dirichlet-to-Neumann operator
    \begin{equation}\label{eq:DtN}
        \mathcal{N}f := n\cdot\nabla\Phi|_{\partial\Omega}
    \end{equation}
    induces an isomorphism from $H^{1/2}(\partial\Omega)/\xR$ onto the subspace of $H^{-1/2}(\partial\Omega)$ with zero mean. However, the Dirichlet-to-Neumann operator is non-local and does not preserve supports. Therefore localized Dirichlet and Neumann controls have to be treated separately.
\end{remark}

\begin{remark}\label{rmk:Distribution}
    For a Lipschitz boundary, we use the notation $C^\infty(\partial\Omega;\xR)$ for restrictions to $\partial\Omega$ of functions in $C^\infty(\xR^2;\xR)$. For a relatively open subset $\Sigma\subset\partial\Omega$, we set
    \begin{equation}\label{eq:SmoothCl_LocalBdy}
        C_c^\infty(\Sigma):=\left\{\varphi|_{\partial\Omega}:\varphi\in C_c^\infty(\xR^2),\ \Supp{\varphi|_{\partial\Omega}}\Subset\Sigma\right\},
    \end{equation}
    and
    \begin{equation}\label{eq:SmoothCl_LocalBdy_ZeroMean}
        C_{c,0}^\infty(\Sigma):=\left\{g\in C_c^\infty(\Sigma):\int_{\partial\Omega}g\di\sigma=0\right\}.
    \end{equation}
    We endow $C_c^\infty(\Sigma)$ with the quotient locally convex topology induced by the ambient test functions. Its continuous dual is denoted by $\mathcal D'(\Sigma)$. The same definition with $\Sigma=\partial\Omega$ gives $\mathcal D'(\partial\Omega)$. Thus, if a weak conormal derivative is said to vanish on $\Sigma$, this means that its duality pairing with every element of $C_c^\infty(\Sigma)$ is zero. Smooth dependence in $C_c^\infty(\Sigma)$ below is always realized in a fixed finite-dimensional subspace and with support in one common compact subset of $\Sigma$.
\end{remark}
Throughout the paper, saying that an arc, or data prescribed along it, is real-analytic on a closed parameter interval means that it extends real-analytically to an open neighborhood of that interval. This convention permits local analytic constructions to be continued slightly beyond the endpoints.

\subsection{Main results}

We now state formal versions of the main theorems to be proved in this paper. Throughout this subsection, $\Omega$ is a domain verifying the hypothesis~\eqref{hyp:Domain}, and $\Sigma\subset\partial\Omega$ is nonempty and relatively open. For every $N\in\xN_+$ under consideration, the nonzero intensities $\kappa_1,...,\kappa_N$ and the circulations $\mathscr C_1,...,\mathscr C_m$ are fixed.

\medskip

The first result is the controllability of the point vortex system for active control; namely, the control function can vary in time. In this case, the controllability is very strong: we can control the point vortices to move along any admissible curves. More precisely, we can prove the following result, where a rough version is stated below and the rigorous statement is given in Theorem~\ref{thm:ActCtrl}.

\begin{theorem}[Active controllability, rough statement]
    Consider arbitrary time $T>0$ and admissible configurations $\bm X^\initial,\bm X^\final\in\mathcal C_N(\Omega)$ (see~\eqref{eq:ConfigSet}). Then, for all prescribed smooth collision-free trajectories $\bm\Gamma=(\Gamma_1,...,\Gamma_N)$ inside $\Omega$ from $\bm X^\initial$ to $\bm X^\final$, there exists a Dirichlet control function $f = f(t)\in C^\infty_c(\Sigma)$, smoothly depending on time, such that $\bm X=\bm\Gamma$ solves~\eqref{eq:KR}-\eqref{eq:CtrlD}.

    Moreover, there also exists a Neumann control $g = g(t)\in C^\infty_c(\Sigma)$, smoothly depending on time, such that $\bm X=\bm\Gamma$ solves~\eqref{eq:KR}-\eqref{eq:CtrlN}.
\end{theorem}

The idea of the proof is quite straightforward. When the control function is allowed to vary in time, the problem can be reduced to: \textit{given $N$ points inside the domain and a vector at each point, can we find a boundary data whose associated background flow coincides with these vectors at the corresponding points?} Furthermore, from Poisson integral, this problem is equivalent to a linear system for the control function. Thus, it suffices to prove the surjectivity of the linear mapping, which sends localized boundary data to the values of the gradient of its harmonic extension at the $N$ vortex positions. This will be justified by contradiction and the unique continuation property (see Proposition~\ref{prop:UniContin}). A finite covering and a partition of unity then produce boundary data depending smoothly on the path parameter.

\medskip

The second result is for stationary control - control independent of time. In this case, we cannot force the point vortices along arbitrary admissible prescribed trajectories if the control is smooth. In Example~\ref{ex:ImpoTraj}, we can see that, the existence of a single point vortex moving along a straight line implies that the boundary data $f$ does not belong to $H^{1/2}(\partial\Omega)$. However, it is still possible to construct an exact control so that the trajectory is close to arbitrary admissible prescribed one. The following is a rough statement (cf.~Theorem~\ref{thm:PassCtrl} for rigorous statement):

\begin{theorem}[Stationary controllability, rough statement]
    Consider any initial and final configurations $\bm X^\initial,\bm X^\final\in\mathcal C_N(\Omega)$ with $X_j^\initial\neq X_j^\final$ for $1\leq j\leq N$, and any mutually disjoint prescribed real-analytic curves $\Gamma_j$ from $X^\initial_j$ to $X^\final_j$. 
    
    Then there exists a time-independent Dirichlet control function $f \in C^\infty_c(\Sigma)$ such that the unique solution to~\eqref{eq:KR}-\eqref{eq:CtrlD} initiating from $\bm X^\initial$ arrives at $\bm X^\final$ for some time. Moreover, the corresponding trajectory from $\bm X^\initial$ to $\bm X^\final$ can be arbitrarily close to the prescribed one, up to re-parametrization.

    Furthermore, there also exists a time-independent Neumann control function $g \in C^\infty_{c,0}(\Sigma)$ such that the unique solution to~\eqref{eq:KR}-\eqref{eq:CtrlN} initiating from $\bm X^\initial$ arrives at $\bm X^\final$ for some time and the corresponding trajectory is close to the prescribed one, up to re-parametrization.

    In both cases, the corresponding background potential $\Phi$ admits a single-valued harmonic conjugate.
\end{theorem}

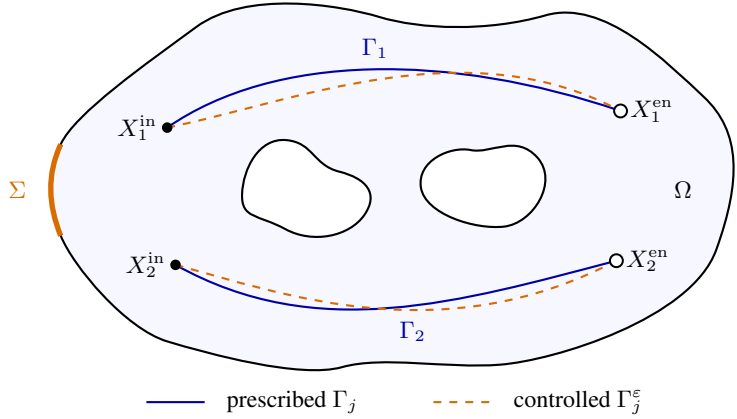
\begin{figure}[htbp]
\centering
\begin{tikzpicture}[scale=1.1, font=\small]
    \draw[thick, fill=blue!3] (-4.03,0.55)
        .. controls (-4.20,0.18) and (-4.18,-0.18) .. (-4.03,-0.55)
        .. controls (-3.88,-0.92) and (-3.25,-1.67) .. (-2.75,-1.82)
        .. controls (-2.25,-1.97) and (-1.15,-2.27) .. (-0.65,-2.13)
        .. controls (-0.15,-1.99) and (0.60,-2.15) .. (1.25,-2.10)
        .. controls (1.90,-2.05) and (3.55,-1.55) .. (3.85,-0.72)
        .. controls (4.15,0.11) and (4.15,0.70) .. (3.70,1.12)
        .. controls (3.25,1.54) and (2.50,2.30) .. (1.80,2.10)
        .. controls (1.10,1.90) and (0.30,1.93) .. (-0.20,2.08)
        .. controls (-0.70,2.23) and (-1.80,2.40) .. (-2.40,1.96)
        .. controls (-3.00,1.52) and (-3.86,0.92) .. (-4.03,0.55);
    \draw[thick, fill=white] (0.30,0.03)
        .. controls (0.25,0.35) and (0.65,0.53) .. (0.90,0.48)
        .. controls (1.15,0.43) and (1.44,0.68) .. (1.70,0.39)
        .. controls (1.96,0.10) and (1.65,-0.40) .. (1.30,-0.44)
        .. controls (0.95,-0.48) and (0.35,-0.29) .. (0.30,0.03);
    \draw[thick, fill=white] (-1.85,-0.05)
        .. controls (-1.84,0.25) and (-1.65,0.63) .. (-1.32,0.60)
        .. controls (-0.99,0.57) and (-1.00,0.30) .. (-0.68,0.22)
        .. controls (-0.36,0.14) and (-0.20,-0.04) .. (-0.38,-0.30)
        .. controls (-0.56,-0.56) and (-1.00,-0.65) .. (-1.33,-0.47)
        .. controls (-1.66,-0.29) and (-1.86,-0.35) .. (-1.85,-0.05);
    \draw[orange!85!black, line width=2pt] (-4.03,0.55)
        .. controls (-4.20,0.18) and (-4.18,-0.18) .. (-4.03,-0.55);
    \node[orange!85!black] at (-4.55,0) {$\Sigma$};
    \node at (3.45,0) {$\Omega$};

    \draw[blue!65!black, thick] (-2.75,0.75)
        .. controls (-1.45,1.65) and (0.65,1.65) .. (2.7,0.95);
    \draw[orange!85!black, dashed, thick] (-2.75,0.75)
        .. controls (-1.25,1.05) and (0.85,1.95) .. (2.7,0.95);
    \draw[blue!65!black, thick] (-2.65,-0.9)
        .. controls (-0.95,-1.85) and (0.75,-1.35) .. (2.65,-0.85);
    \draw[orange!85!black, dashed, thick] (-2.65,-0.9)
        .. controls (-1.35,-1.25) and (0.55,-1.95) .. (2.65,-0.85);
    \fill (-2.75,0.75) circle (1.8pt)
        node[left] {$X_1^\initial$};
    \draw[fill=white] (2.7,0.95) circle (2.2pt)
        node[right] {$X_1^\final$};
    \fill (-2.65,-0.9) circle (1.8pt)
        node[left] {$X_2^\initial$};
    \draw[fill=white] (2.65,-0.85) circle (2.2pt)
        node[right] {$X_2^\final$};
    \node[blue!65!black] at (-0.25,1.7) {$\Gamma_1$};
    \node[blue!65!black] at (0.2,-1.7) {$\Gamma_2$};

    \draw[blue!65!black, thick] (-3,-2.55) -- (-2.3,-2.55);
    \node[anchor=west] at (-2.15,-2.55) {prescribed $\Gamma_j$};
    \draw[orange!85!black, dashed, thick] (0.45,-2.55) -- (1.15,-2.55);
    \node[anchor=west] at (1.3,-2.55) {controlled $\Gamma_j^\varepsilon$};
\end{tikzpicture}
\caption{Stationary control for two point vortices.
The prescribed disjoint arcs $\Gamma_j$ (solid) and the controlled
trajectories $\Gamma_j^\varepsilon$ (dashed) have the same endpoints
$X_j^\initial,X_j^\final$.}
\label{fig:stationary-arcs}
\end{figure}

The mutual disjointness of prescribed curved exclude certain configurations such as permutation. It is still unknown whether permutation is realizable with stationary control. And, as we can see in the proof, the control time can be arbitrarily small. The existence of harmonic conjugate for the background potential $\Phi$ is not required in the proof, while it guarantees that the controlled system~\eqref{eq:KR} is Hamiltonian, as is the uncontrolled system; see Remark~\ref{rmk:Hamilton} for the Hamiltonian formulation.

\medskip

The main idea of the proof is that, if we choose background flow $\nabla\Phi$ large enough, the effect of other terms on the right-hand side of~\eqref{eq:KR} is negligible. By ignoring these terms, the local solvability of Laplace equation (see Lemma~\ref{lem:CurveData_HarFct}) gives a harmonic function whose gradient induces exactly the desired trajectory. To make the proof rigorous, there are two main difficulties: (1) the ignored terms, which are analytic and much smaller than the background flow, are not identically zero generally; (2) the harmonic function from the local solvability is only defined locally near the prescribed curves. 

\medskip

To deal with the second problem, we will apply a refined Runge approximation theory, in order to approximate the locally defined harmonic function by those defined on the whole space $\Omega$ with desired boundary data. This gives the approximate controllability. The error comes from both the ignored terms mentioned in the first problem and the Runge approximation. By adding some well-chosen small correction terms on the harmonic background flow, we can apply the implicit function theorem to pass to the exact control. And since the corrections are small, the resulting trajectory can be arbitrarily close to the prescribed ones.

\medskip

With these controllability results, it is natural to consider the optimal control problem. In this paper, we consider the cost function taking the form:
\begin{equation*}
    J_{\mu,\beta}(f,T) = \sum_{j=1}^N \int_0^T|\dot X_j(t)|\di t + \beta T + \mu\|f\|_{H^s(\partial\Omega)}^2.
\end{equation*}
Here, three sources of penalization are considered: the trajectory length, the control time, and the size of control function $f$. It is possible to show that, among all admissible controls between given configurations, the infimum of this cost function can be attained. A rough result is as follows and we refer to Theorem~\ref{thm-existence-optimal-N} for the precise statement.

\begin{theorem}[Existence of an optimal control, rough version]
    Let $s,\beta,\mu>0$. Given any initial and final configuration $\bm X^\initial,\bm X^\final$ away from boundary, if there exists localized control from $\bm X^\initial$ to $\bm X^\final$, then there exists an optimal control  $(f_\ast,T_\ast)$ such that
    \begin{equation*}
        J_{\mu,\beta}(f_\ast,T_\ast)
        =\inf J_{\mu,\beta}(f,T).
    \end{equation*}
    The analogous assertion holds for compatible stationary Neumann controls $g\in H^{s-1}(\partial\Omega)$, with the control term in the cost replaced by $\mu\|g\|_{H^{s-1}(\partial\Omega)}^2$.
\end{theorem}

The proof is direct. For a minimizing sequence, we can assume the weak convergence of control function $f$, whose associated harmonic background flow has strong convergences inside the domain. Together with the uniform boundedness of control time, this implies the convergence of the trajectory length. Thus the minimizer exists as a limit of some minimizing subsequence. 

\medskip

We also consider the case where only the control length is penalized, namely, with $\beta=\mu=0$. For the single point vortex system ($N=1$), the functional $J_{0,0}$ defines the optimal controlled length. Recalling from the active control, we can realize a control close to any prescribed real-analytic curves. Consequently, we can show that this optimal controlled length is identical to the intrinsic distance in the domain $\Omega$ (see Theorem~\ref{thm:OptCtrlLen}). 

\medskip

However, without penalization in control time and control function, this optimal controlled length cannot be attained by regular control functions, even for regular convex domains. A counterexample of unit disk is given in Example~\ref{ex:ImpoTraj} and the same result for general smooth convex domain is given in Theorem~\ref{thm-regularity-obstruction-convex}, where we can show that the infimum can only be attained by control functions $f\in H^{s}(\partial\Omega)$ with $s<1/2$.

\medskip

The paper is organized as follows. Section~\ref{sect:ActCtrl} proves active trajectory controllability by reducing the control problem to harmonic interpolation at finitely many moving points. Section~\ref{sect:HarmonicApproximation} develops the two approximation tools needed for stationary controls: local harmonic realization along analytic arcs and localized Runge approximation from the prescribed boundary patch. Section~\ref{sect:PassCtrl} combines these tools with a finite-dimensional endpoint correction and a time rescaling to obtain stationary endpoint controllability. Section~\ref{sect:OptCtrl} studies the associated optimization problems, first by the direct method for a regularized cost and then through the intrinsic geometry of the domain in the low regular one-vortex problem. The technical elliptic results used throughout the paper are collected in Appendix~\ref{app:Tech}.

\hspace{1em}

\noindent\textbf{Declaration of AI use.} During the preparation of this manuscript, the authors used ChatGPT (OpenAI) to assist with language editing, improve the exposition, and discuss and refine selected mathematical arguments. The authors take full responsibility for the accuracy and originality of the final manuscript, including all results, proofs, and references.

\section{Active boundary control}\label{sect:ActCtrl}

In this section, we prove that the system~\eqref{eq:KR} is controllable with the time-dependent Dirichlet control~\eqref{eq:CtrlD} or Neumann control~\eqref{eq:CtrlN} \textit{in the strong sense}. The rigorous statement of the result is as follows:
\begin{theorem}[Active controllability]\label{thm:ActCtrl}
    Assume that $\Omega\subset\xR^2$ is a domain satisfying the hypothesis~\eqref{hyp:Domain}. Consider an integer $N\in\xN_+$ and a real number $T>0$. Let $\bm{X}^\initial = (X_1^\initial,...X_N^\initial)$ and $\bm{X}^\final = (X_1^\final,...,X_N^\final)$ be two configurations of $N$ point vortices in $\Omega$, belonging to the set of configurations $\mathcal{C}_N(\Omega)$ defined in~\eqref{eq:ConfigSet}.
    
    Let $\Sigma\subset\partial\Omega$ be any nonempty relatively open subset. Then, for any smooth path $\bm{\Gamma} = (\Gamma_1,...\Gamma_N) \in C^\infty([0,T];\Omega^N)$ satisfying
    \begin{equation*}
        \bm{\Gamma}(0) = \bm{X}^\initial,\quad \bm{\Gamma}(T) = \bm{X}^\final, \quad \inf_{\substack{t\in[0,T] \\ 1\leq j<k\leq N}}|\Gamma_j(t) - \Gamma_k(t)| > 0, \quad \inf_{\substack{t\in[0,T] \\ 1\leq j\leq N}} \dist(\Gamma_j(t), \partial\Omega) > 0,
    \end{equation*}
    there exists a smooth control $f \in C^\infty([0,T]\times\partial\Omega;\xR)$ satisfying
    $$
        \Supp{f(t,\cdot)}\subset \Sigma,\qquad \forall t\in[0,T],
    $$
    such that $\bm{\Gamma}$ is a solution to the controlled point vortex system~\eqref{eq:KR} with control on Dirichlet condition~\eqref{eq:CtrlD}.
    
    Similarly, there also exists a smooth control $g \in C^\infty([0,T]\times\partial\Omega;\xR)$ satisfying
    $$
        \Supp{g(t,\cdot)}\subset \Sigma,
        \qquad \int_{\partial\Omega}g(t,\cdot)\di\sigma=0,
        \qquad \forall t\in[0,T],
    $$
    such that $\bm{\Gamma}$ is a solution to the controlled point vortex system~\eqref{eq:KR} with control on Neumann condition~\eqref{eq:CtrlN}.
\end{theorem}

This active control problem can be reduced to a pointwise-in-time problem. Indeed, once the collision-free trajectory $\bm \Gamma$ has been prescribed, the uncontrolled velocity of every vortex is known along the trajectory. The remaining task is therefore to construct, at each time, a harmonic potential whose gradient takes finitely many prescribed values at the vortex positions. The essential point is not only the solvability of this interpolation problem, but also the possibility of choosing the boundary data smoothly in time and with support contained in one fixed compact subset of the control region.

\subsection{Harmonic interpolation}

This part is devoted to the study of the following time-independent problem:
\begin{problem}\label{prob-glb:HarFct}
    For a configuration $(X_1,...,X_N)\in\mathcal C_N(\Omega)$ and vectors $v_1,...,v_N\in\xR^2$, can we find a harmonic function $\Phi$ such that $\nabla\Phi(X_j) = v_j$ for all $j=1,...,N$? If yes, can we choose $\Phi$ so that it depends smoothly on time when the points and vectors do, and so that its boundary data is supported in any given nonempty relatively open subset $\Sigma\subset\partial\Omega$?
\end{problem}
This interpolation statement isolates the elliptic ingredient of the active control argument. It says that localized boundary data generate all possible collections of gradient values at finitely many distinct interior points. The Dirichlet and Neumann formulations require separate arguments because the Neumann datum must have zero mean, but both are ultimately consequences of boundary unique continuation (see Proposition~\ref{prop:UniContin}).

\begin{example}\label{ex:GlobalCtrl}
    Let us start with the simplest case $\Sigma=\partial\Omega$. For the sake of simplicity, we use complex variables, and the problem above can be reduced to the following: Given pairwise distinct points $z_1,...,z_N\in\xC$ and arbitrary values $w_1,...,w_N\in\xC$, there exists a polynomial function $F = F(z)$ of degree at most $N$ such that
    $$
        F'(z_j) = w_j, \quad \forall 1 \le j \le N.
    $$
    Indeed, if $v_j=(v_{j,1},v_{j,2})$, we take $w_j=v_{j,1}-iv_{j,2}$ and $\Phi=\operatorname{Re}F$. Then $\Phi$ is real-valued and harmonic, and $F'(z_j)=w_j$ is equivalent to $\nabla\Phi(z_j)=v_j$.
    Moreover, the coefficients of $F$ depend smoothly on $(z_1,...z_N,w_1,...w_N)$.
    By denoting the polynomial $F(z)$ as 
    $$
        F(z) = \sum_{j=0}^{N-1} \frac{\alpha_j}{j+1} z^{j+1},
    $$
    the desired relation $F'(z_k)=w_k$ becomes the following linear system for $\alpha_0,...,\alpha_{N-1}$:
    \begin{equation*}
        \sum_{j=0}^{N-1} z_k^j \alpha_j = w_k, \quad 1\le k\le N,
    \end{equation*}
    whose determinant is Vandermonde polynomial
    $$
        \prod_{j<k}(z_k-z_j) \neq 0.
    $$
    Since $z_j$'s are mutually distinct, $\alpha_j$'s can be expressed as rational functions of $z_1,...z_N$ and $w_1,...w_N$. 
\end{example}
If the control is assumed to be applied only on some subsets of the boundary, the reduction in Example~\ref{ex:GlobalCtrl} should be improved to the following proposition.
\begin{proposition}\label{prop:HarFct}
    Consider a bounded and connected domain $\Omega\subset\xR^2$ with Lipschitz boundary having finitely many connected components. Let $N$ be a positive integer and $\Sigma$ be a nonempty relatively open subset of $\partial\Omega$. Then, for every $(X_1,...,X_N)\in \mathcal{C}_N(\Omega)$ (see~\eqref{eq:ConfigSet}) and arbitrary vectors $v_1,...,v_N\in\xR^2$, there exists $f\in C_c^\infty(\Sigma)$ whose real-valued harmonic extension $\Phi$ satisfies
    $$
        \nabla \Phi(X_j) = v_j, \quad \forall j=1,...N.
    $$
    Moreover, when $X_j$'s and $v_j$'s depend smoothly on some parameter $t\in [0,T]\subset \xR$ and the points remain mutually separated and in the interior, the boundary datum $f=\Phi|_{\partial\Omega}$ may be chosen in $C^\infty([0,T];C_c^\infty(\Sigma))$.
    
    The same assertions hold for a Neumann datum
    $$
        g:=n\cdot\nabla\Phi|_{\partial\Omega}\in C_{c,0}^\infty(\Sigma).
    $$
    In the parameter-dependent case, $g$ can be chosen in $C^\infty([0,T];C_{c,0}^\infty(\Sigma))$.
\end{proposition}

This proposition can be proved in the same spirit as Example~\ref{ex:GlobalCtrl}. We will reduce the equation $\nabla \Phi(X_j) = v_j$ to a linear system and then prove the solvability (not necessarily unique) of such system. The proof has two parts. We first prove surjectivity for each fixed configuration by showing that every linear functional annihilating the attainable gradient values must vanish. We then use a finite covering of the compact parameter set and a partition of unity to obtain a family of controls which depends smoothly on the parameter. This second step also yields one common compact support in $\Sigma$, a property needed when the configuration varies along a prescribed trajectory. In what follows, we focus on the Dirichlet condition, while the proof of the Neumann case will be sketched at the end.

\hspace{1em}

\noindent\textbf{Step 1: Reduction to linear system.} Let $f$ be the trace of $\Phi$ at the boundary $\partial\Omega$. In terms of the Poisson kernel $P(X,Y)$, Green's representation formula reads
$$
    \Phi(X)
    =\langle P(X,\cdot),f\rangle_{\partial\Omega}
    =\left\langle n\cdot\nabla_YG_\Omega(X,\cdot),f\right\rangle_{\partial\Omega}.
$$
Recall that $n$ is the outward normal direction. Since the boundary $\partial\Omega$ is Lipschitz, the identity
$$
    P(X,\cdot)=n\cdot\nabla_YG_\Omega(X,\cdot)
$$
holds in the distributional sense for $X$ in the interior of $\Omega$ (cf.~\cite[Theorem 3]{dahlberg1977harmonic} and Remark~\ref{rmk:Distribution}). In particular, we have
$$
    \left\langle (n\cdot\nabla_Y)\nabla_{X}G_\Omega(X,\cdot),f\right\rangle_{\partial\Omega}
    :=\nabla\Phi(X).
$$
The problem can thus be reduced to the following linear system for $f$,
\begin{equation}\label{eq:LinSys}
    \nabla\Phi(X_j)
    =\left\langle(n\cdot\nabla_Y)\nabla_XG_{\Omega}(X_j,\cdot),f\right\rangle_{\partial\Omega}
    =v_j,
    \qquad j=1,...,N.
\end{equation}
Here one should understand the components of $(n\cdot\nabla_Y)\nabla_X G_{\Omega}(X_j,\cdot)$ as distributions on $\partial\Omega$ and $f$ is a smooth function on $\partial\Omega$ to be determined.

\hspace{1em}

\noindent\textbf{Step 2: Solvability of the linear system.} The solvability of the linear system~\eqref{eq:LinSys} above can be summarized as the following abstract lemma:
\begin{lemma}\label{lem:Resol_Lin_Sys}
    Consider a bounded connected domain $\Omega$ and a nonempty relatively open subset $\Sigma\subset\partial\Omega$. Assume that
    \begin{itemize}
        \item $\{w_k\}_{k=1}^M$ are real-valued distributions on $\Sigma$ depending smoothly on the parameter $t\in[0,T] \subset\xR$ in the sense that
        $$
            \forall \varphi\in C^\infty_c(\Sigma), \quad t\mapsto \langle w_k(t), \varphi \rangle \text{ is smooth}.
        $$
        \item $\{c_k(t)\}_{k=1}^M$ are constants depending smoothly on the parameter $t$.
        \item $w_k(t)$'s are linearly independent for all $t\in[0,T]$. More precisely, for all $t\in [0,T]$, there holds
        $$
            \forall (\lambda_1,...\lambda_M) \in \xR^M, \quad \sum_{k=1}^M \lambda_k w_k(t) = 0 \text{ in }\mathcal{D}'(\Sigma) \Leftrightarrow (\lambda_1,...\lambda_M)=0.
        $$
    \end{itemize}
    Then there exists at least one function $f\in C^\infty([0,T];C_c^\infty(\Sigma))$, with $f(t)$ supported in $\Sigma$ for all time $t\in[0,T]$, such that
    $$
        \langle w_k(t), f(t) \rangle = c_k(t), \quad k=1,...M.
    $$
\end{lemma}
\begin{proof}
    To begin with, we freeze the parameter $t = t_0 \in [0,T]$. Consider the linear mapping
    \begin{equation*}
        C_c^\infty(\Sigma) \to \xR^M, \quad \varphi  \mapsto (\langle w_1(t_0), \varphi \rangle,...\langle w_M(t_0), \varphi \rangle),
    \end{equation*}
    which is surjective due to the assumption of linear independence. Thus, there exists at least one set of compactly supported smooth functions $\{\varphi_k\}_{k=1}^M$ such that the matrix
    $$
        M(t_0) := \Big( \langle w_k(t_0), \varphi_{k'} \rangle \Big)_{k,k'} \in \GL_M(\xR).
    $$
    Since $\GL_M(\xR)$ forms an open set in the space of matrices $\xR^{M\times M}$, there exists $\epsilon>0$ such that $M(t)$ is invertible for all $t\in [0,T] \cap ]t_0-\epsilon,t_0+\epsilon[$. Moreover, the mapping $t \mapsto M(t)^{-1}$ is smooth from $[0,T] \cap ]t_0-\epsilon,t_0+\epsilon[$ to $\GL_M(\xR)$.
    As a conclusion, for all $t$ in the small neighborhood $ [0,T] \cap ]t_0-\epsilon,t_0+\epsilon[$ of $t_0$, the desired $f(t)$ can be chosen as
    $$
        f(t) := \sum_{k=1}^M \alpha_k(t) \varphi_k, \quad (\alpha_k) := M(t)^{-1}(c_k).
    $$
    Clearly, the smooth dependence on $t$ of $M(t)^{-1}$ and the functions $c_k(t)$ implies that of $f(t)$.

    It remains to explain how to obtain the desired function $f(t)$ for all $t\in[0,T]$. Indeed, through the argument above, given any $t\in[0,T]$, we can construct the desired $f$ on a neighborhood $U(t)$ of $t$. The compactness of $[0,T]$ guarantees a finite cover $U(t_l)$ ($l=1,...L$), corresponding to smooth functions $f_l(t)$. Let us denote by $\chi_l$ a smooth partition of unity subordinate to this finite cover, chosen with $\Supp{\chi_l}\Subset U(t_l)$ relative to $[0,T]$. We extend $\chi_lf_l$ by zero outside $U(t_l)$ and define the global smooth function by
    $$ 
        f(t) := \sum_{l=1}^L \chi_l(t) f_l(t).
    $$
    A direct calculation gives
    $$
        \langle w_k(t), f(t) \rangle = \sum_{l=1}^L \chi_l(t) \langle w_k(t), f_l(t) \rangle =  \sum_{l=1}^L \chi_l(t) c_k(t) = c_k(t).
    $$
    Each $f_l$ takes values in the span of finitely many fixed test functions. Hence $f$ is smooth in $t$ and all $f(t)$ are supported in the compact union of the supports of these finitely many test functions.
\end{proof}

\hspace{1em}

\noindent\textbf{Step 3: Linear independence of the coefficients.} For every $f\in C^\infty(\partial\Omega)$, the duality pairing in~\eqref{eq:LinSys} equals $\nabla\Phi(X)$ and is therefore smooth in $X\in\Omega$ by interior regularity. Thus the coefficient distributions depend smoothly on the configuration in the sense required by Lemma~\ref{lem:Resol_Lin_Sys}. It remains to check their linear independence.

Given vectors $c_1,...c_N\in\xR^2$, we define the auxiliary distribution
$$
    V(Y) := \sum_{j=1}^N c_j\cdot\nabla_X G_{\Omega}(X_j,Y).
$$
The goal is to show that $n\cdot\nabla V$ vanishes on $\Sigma$ only if $c_j$'s are all zero. 
We first notice that $V$ is harmonic inside $\Omega$ except at the points $X_j$'s according to the definition of Green function. Since the Green function $G_\Omega$ is associated with the Dirichlet boundary condition, $G_{\Omega}(X,Y)$ equals zero for all $X\in\Omega$ and $Y\in\partial\Omega$. Thus, $V$ vanishes on the full boundary $\partial\Omega$, in particular on $\Sigma$.
Once the weak conormal derivative $n\cdot \nabla V$ also vanishes on $\Sigma$, choose pairwise disjoint closed balls around the points $X_j$, small enough that their complement in $\Omega$ is connected. The function $V$ belongs to $H^1$ on this punctured Lipschitz domain and has zero Dirichlet and conormal traces on $\Sigma$. Proposition~\ref{prop:UniContin} shows that $V$ vanishes there. Interior unique continuation in the connected punctured domain $\Omega\backslash\{X_1,...,X_N\}$ then shows that $V$ vanishes throughout that set. Recalling that $G_\Omega$ admits the decomposition into singular and regular parts,
$$
    G_\Omega(X,Y) = \frac{1}{2\pi}\ln|X-Y| + \tilde{R}_\Omega(X,Y).
$$
Then, near $X_j$, the behavior of $V$ is governed by its singular part:
$$
    V(Y) = \frac{1}{2\pi} \frac{c_j \cdot (X_j-Y)}{|X_j-Y|^2} + O(1), \quad \text{as }Y\to X_j.
$$
Since $V$ vanishes near $X_j$, one has immediately that $c_j=0$. This proves the desired linear independence. Lemma~\ref{lem:Resol_Lin_Sys} therefore proves the Dirichlet part of Proposition~\ref{prop:HarFct}, including the asserted smooth dependence.

\hspace{1em}

\noindent\textbf{Sketch of proof: Neumann case.} Since the proof of this proposition is almost identical to the Dirichlet case, we only highlight the main differences for the sake of simplicity.
We fix the additive constant in the potential by imposing
$$
    \int_\Omega\Phi\di X=0.
$$
The representation~\eqref{eq:LinSys} of the harmonic function should be replaced by
$$
    \nabla \Phi(X_j) = -\left\langle\nabla_X G_\Omega^N(X_j,\cdot),g\right\rangle_{\partial\Omega} = v_j, \quad j=1,...,N,
$$
where $g\in C^\infty_{c,0}(\Sigma)$ and $G_\Omega^N$ is the normalized Neumann Green function, defined by
$$
    \begin{aligned}
        &\Delta_Y G_\Omega^N(X,Y)=\delta_X(Y)-\frac{1}{|\Omega|},
        \qquad n\cdot\nabla_YG_\Omega^N(X,Y)|_{Y\in\partial\Omega}=0,\\
        &\int_\Omega G_\Omega^N(X,Y)\di Y=0.
    \end{aligned}
$$
Recall the space $C^\infty_{c,0}(\Sigma)$ defined in~\eqref{eq:SmoothCl_LocalBdy_ZeroMean}. In order to apply Lemma~\ref{lem:Resol_Lin_Sys}, we turn to consider the linear combination 
$$
    W(Y) := \sum_{j=1}^N c_j \cdot\nabla_X G^N_\Omega(X_j,Y).
$$
The function $W$ is harmonic away from the points $X_j$'s. And its normal derivative vanishes on the full boundary:
$$
    n\cdot\nabla W = \sum_{j=1}^N (c_j \cdot\nabla_X) (n\cdot\nabla_Y) G^N_\Omega(X_j,Y) = 0.
$$
If $W$ vanishes on $\Sigma$, the zero-extension argument in the proof of Proposition~\ref{prop:UniContin}, applied to $W$, shows that $W=0$ in an interior neighborhood of this boundary patch. Interior unique continuation gives $W=0$ except at $X_j$'s. Note that the normalized Neumann Green function has the same local logarithmic singularity as $G_\Omega$. The singularity at $X_j$ then forces $c_j=0$ for every $j$. 

An application of Lemma~\ref{lem:Resol_Lin_Sys} then gives a smooth function $\tilde{g} \in C^\infty([0,T];C^\infty_c(\Sigma))$ whose mean value on $\partial\Omega$ is not necessarily zero. To deal with this issue, we turn to consider the function
$$
    g(t) : \tilde{g}(t) - \int_{\partial\Omega}\tilde{g}(t)\di\sigma,
$$
which belongs to the zero-mean space $C^\infty([0,T];C_{c,0}^\infty(\Sigma))$. Then the zero-mean assumption on $G^\Neu_\Omega$ gives the desired equality
$$
    -\left\langle\nabla_X G_\Omega^N(X_j,\cdot),g(t)\right\rangle_{\partial\Omega} = -\left\langle\nabla_X G_\Omega^N(X_j,\cdot),\tilde{g}(t)\right\rangle_{\partial\Omega} = v_j.
$$

\subsection{Proof of active controllability}
We now return to the point-vortex dynamics. Along the prescribed trajectory, the difference between the desired velocity and the uncontrolled point-vortex drift determines the gradient values that the control potential must realize. Since the trajectory remains in a compact subset of the collision-free configuration space, these target values depend smoothly on time and the interpolation points remain uniformly distinct and uniformly separated from the boundary.
For $1\leq j\leq N$, define
\begin{equation}\label{eq:ActiveTargetVector}
    \begin{aligned}
        v_j(t):={}&\dot{\Gamma}_j(t)
        -\sum_{k \neq j}\kappa_k\nabla_X^\perp G_\Omega(\Gamma_j(t),\Gamma_k(t))\\
        &-\frac{\kappa_j}{2}\nabla^\perp R_\Omega(\Gamma_j(t))
        -\sum_{r=1}^m c_r(\bm\Gamma(t))\nabla^\perp W_r(\Gamma_j(t)).
    \end{aligned}
\end{equation}
Using uniqueness in the Dirichlet case, and uniqueness of the gradient in the Neumann case, it suffices to find a harmonic function $\Phi = \Phi(t,X)$ such that
$$
    \nabla\Phi(t,\Gamma_j(t))=v_j(t),
    \qquad 1\leq j\leq N.
$$
This is precisely the harmonic interpolation problem solved by Proposition~\ref{prop:HarFct}.

We now complete the proof of Theorem~\ref{thm:ActCtrl}. The vectors $v_j$ defined in~\eqref{eq:ActiveTargetVector} depend smoothly on $t$ because the prescribed path remains in a compact subset of $\mathcal C_N(\Omega)$. Applying Proposition~\ref{prop:HarFct} to $\Gamma_j(t)$ and $v_j(t)$ gives, respectively, the asserted Dirichlet and Neumann controls, with one common compact support in $\Sigma$. Substitution in~\eqref{eq:KR} shows that $\bm X=\bm\Gamma$, and proves the theorem.

\section{Harmonic approximation and Runge theory}\label{sect:HarmonicApproximation}

The construction of stationary controls requires a different mechanism. A single time-independent harmonic potential must generate the prescribed motion along an entire arc, and therefore the values of its gradient cannot be chosen independently at different times. We bridge this difficulty in two stages. First, analytic Cauchy data are realized by a harmonic function in a neighborhood of each prescribed arc. Second, localized Runge approximation replaces these local functions by global harmonic potentials whose boundary data are supported in $\Sigma$. Together with a perturbative argument, we can correct this approximation and obtain an exact control.

\subsection{Local harmonic realization of analytic arcs}

The local construction is geometric. Along a regular analytic arc, the tangential component of the desired gradient determines the trace of the potential up to an additive constant, while the normal component determines its normal derivative. The arc is noncharacteristic for the Laplace operator, so the Cauchy-Kovalevskaya theorem produces a harmonic potential realizing both components. Note that the use of arcs, rather than closed curves, avoids an additional period condition.

\begin{lemma}[Local solvability]\label{lem:CurveData_HarFct}
    Let $\Gamma:[0,1]\to\xR^2$ be a regular embedded real-analytic arc and let $V$ be a real-analytic vector field along this arc. Then there exists a harmonic function $H$ defined in a neighborhood of $\Gamma([0,1])$ such that
    $$
        \nabla H(\Gamma(t)) = V(\Gamma(t)), \quad \forall t\in[0,1].
    $$
    Moreover, such a harmonic function is unique up to a constant after the neighborhood is chosen connected and sufficiently thin.
\end{lemma}
\begin{proof}
    If such a harmonic function exists, there must hold
    $$
        \frac{\di}{\di t} H(\Gamma(t)) = \dot{\Gamma}(t) \cdot \nabla H(\Gamma(t)) = \dot{\Gamma}(t) \cdot V(\Gamma(t)).
    $$
    Extend the analytic arc and the analytic data slightly beyond both endpoints, and choose an analytic unit normal $n$ along the extended arc. The desired function $H$ must solve the following Cauchy problem in a small neighborhood of $\Gamma$,
    \begin{equation*}
        \left\{\begin{aligned}
            &\Delta H = 0, \\
            &H(\Gamma(t)) = C + \int_0^t \dot{\Gamma}(t') \cdot V(\Gamma(t')) \di t', \\
            &n\cdot\nabla H(\Gamma(t)) = n\cdot V(\Gamma(t)),
        \end{aligned}\right.
    \end{equation*}
    where $C\in\xR$ is arbitrary. Since the arc is noncharacteristic for the Laplace operator and all the data are analytic, the Cauchy-Kovalevskaya theorem gives a local solution. The solutions obtained in overlapping charts agree by uniqueness, and therefore glue along the compact arc. This also proves uniqueness up to the constant $C$, since $\Gamma([0,1])$ is not a closed curve.
\end{proof}

\subsection{Localized Runge approximation}\label{sec:Runge}

Runge approximation transfers the preceding local construction to the admissible global control class. Roughly speaking, the Runge approximation theorem states that locally defined harmonic functions can be approximated by ones defined in a larger domain. The localization requirement is essential to the stationary control problem: the approximating potentials must be generated by Dirichlet or compatible Neumann data supported in the prescribed set $\Sigma$, rather than by arbitrary boundary data.

\begin{definition}\label{def:Runge}
    Consider a bounded domain $\Omega$ and an open set $\Omega_0\Subset\Omega$, such that
    $$
        \Omega\backslash\overline{\Omega_0}\text{ is connected},
    $$
    and a class 
    $$
        S \subseteq \{ U\in H^1(\Omega): \Delta U = 0 \text{ in $\Omega$} \}
    $$
    of harmonic functions in $H^1(\Omega)$. We say that $S$ has the \textit{Runge approximation property} in $\Omega_0$ if, for every function $U$ harmonic in $\Omega_0$, every compact set $K\Subset\Omega_0$, every $k\in\xN$ and every $\epsilon>0$, there exists $U^\epsilon \in S$ such that
    $$
        \| U - U^\epsilon \|_{C^k(K)} < \epsilon.
    $$
\end{definition}

\begin{remark}
    Among many references, Runge approximation theorem is stated for $L^\infty(K)$ norm or $L^2(K)$ norm. We can deduce the $C^k$ approximation above through the interior regularity of harmonic functions (see for example Section 6.3 of~\cite{evans2010pde}).
\end{remark}

The classical Runge approximation theorem states that the class of harmonic functions has the Runge approximation property above. In a recent paper of Lassas, Liimatainen, and Salo~\cite{LassasLiimatainenSalo}, the Runge approximation property is proved for harmonic functions with boundary data supported in a prescribed open subset of the boundary, which we recall here:
\begin{proposition}[Proposition A.2 of~\cite{LassasLiimatainenSalo}]
    Let $\Omega\subset\mathbb{R}^2$ be a bounded connected domain with sufficiently smooth boundary, let $\Omega_0\Subset\Omega$ be an open set with finitely many connected components having sufficiently smooth boundaries such that $\Omega\backslash\overline{\Omega_0}$ is connected, and let $\Sigma\subset\partial\Omega$ be a nonempty relatively open subset. Then the class
    $$
        \{ U\in H^1(\Omega): \Delta U = 0 \text{ in $\Omega$},\ U|_{\partial\Omega}\in C^\infty_c(\Sigma) \}
    $$
    has the Runge approximation property in $\Omega_0$.
\end{proposition}
We also refer to another paper of R\"{u}land and Salo~\cite{ruland2019runge}, where the authors proved the same result for second-order elliptic equations. 

For our purpose, we need to approximate the locally defined harmonic functions by those globally defined in $\Omega$ with (1) boundary data supported in a prescribed relatively open subset $\Sigma$ of the boundary $\partial\Omega$, (2) a single-valued harmonic conjugate. This refinement can actually be proved by a similar argument as in~\cite{LassasLiimatainenSalo}. Before presenting the proof, we recall the boundary condition for the existence of single-valued harmonic conjugate. 

For $(m+1)$-connected domain $\Omega$ satisfying hypothesis~\eqref{hyp:Domain}, recall the harmonic measures $W_j$ introduced in Section~\ref{sect:Intro},
\begin{equation*}
    \Delta W_j = 0, \quad \text{in } \Omega; \qquad W_j|_{\Sigma_j} = 1, \quad W_j|_{\partial\Omega\backslash\Sigma_j} = 0.
\end{equation*}
We have the identity
\begin{equation}\label{eq:HarMes_Unity}
    \sum_{j=0}^m W_j \equiv 1.
\end{equation}
The following lemma provides an equivalent characterization of the existence of a harmonic conjugate.
\begin{lemma}\label{lem:Exact}
    Consider a domain $\Omega\subset\xR^2$ satisfying hypothesis~\eqref{hyp:Domain}. Then, for an arbitrary harmonic function $\Phi \in H^1(\Omega)$, the following assertions are equivalent:
    \begin{enumerate}[label=(\arabic*)]
        \item $\Phi$ admits a harmonic conjugate in $H^1(\Omega)$;
        \item There holds
        \begin{equation*}
            \left\langle n\cdot\nabla \Phi|_{\partial\Omega}, W_j|_{\partial\Omega} \right\rangle_{\partial\Omega}=\int_{\Sigma_j}n\cdot\nabla\Phi\di\sigma = 0, \quad j=1,...m;
        \end{equation*}
        \item There holds 
        \begin{equation*}
            \left\langle n\cdot\nabla W_j|_{\partial\Omega}, \Phi|_{\partial\Omega}\right\rangle_{\partial\Omega} = \int_{\partial\Omega} \Phi n\cdot\nabla W_j\di\sigma =0, \quad j=1,...m.
        \end{equation*}
    \end{enumerate}
    Here $\langle\cdot,\cdot\rangle_{\partial\Omega}$ denotes the $H^{-1/2}(\partial\Omega)$-$H^{1/2}(\partial\Omega)$ duality pairings. Note that, from the identity~\eqref{eq:HarMes_Unity}, the above identities hold automatically with $j=0$ if they are true for $j=1,...m$.
\end{lemma}

We now state the localized Runge approximation theorem used in the proof of Theorem~\ref{thm:PassCtrl}.
\begin{theorem}\label{thm:Runge}
    Let $\Omega\subset\xR^2$ be a domain satisfying hypothesis~\eqref{hyp:Domain}. Let $\Omega_0\Subset\Omega$ be such a domain with $\Omega\backslash\overline{\Omega_0}$ connected, and let $\Sigma\subset\partial\Omega$ be a nonempty relatively open subset. Then the following classes
    \begin{align}
        S^\Dir :=& \left\{U\in H^1(\Omega)\left|\ 
            \begin{aligned}
                &\Delta U=0\text{ in }\Omega,\ U|_{\partial\Omega}\in C_c^\infty(\Sigma), \\
                &\text{$U$ admits a harmonic conjugate in $H^1(\Omega)$}
            \end{aligned}\right.
            \right\}, \label{eq:RungeClDir} \\
        S^\Neu :=& \left\{ U\in H^1(\Omega)           \left|\ \begin{aligned}
                &\Delta U = 0 \text{ in }\Omega,\ n\cdot\nabla U|_{\partial\Omega}\in C_{c,0}^\infty(\Sigma),\\
                &\text{$U$ admits a harmonic conjugate in $H^1(\Omega)$}
            \end{aligned}\right.\right\} \label{eq:RungeClNeu}
    \end{align}
    both have the Runge approximation property in $\Omega_0$ (see Definition~\ref{def:Runge}).
\end{theorem}

The proof is based on duality. Starting from a function orthogonal to all restrictions of admissible global potentials, we solve an adjoint Poisson problem with source supported in $\Omega_0$. Orthogonality forces one part of the Cauchy data of the adjoint solution to vanish, or to be constant in the Neumann case, on $\Sigma$. Boundary unique continuation then propagates this information through the connected set $\Omega\backslash\overline{\Omega_0}$. Green's formula across the interface finally shows that the original function is orthogonal to every locally harmonic target. This gives density in $L^2(\Omega_0)$, and interior elliptic estimates upgrade the approximation to every $C^k$ norm on compact subsets. The boundary unique-continuation result invoked in this argument is Proposition~\ref{prop:UniContin} in Appendix~\ref{app:Tech}.

\begin{proof}
    By interior estimates, it suffices to prove density of $S^\Dir$ and $S^\Neu$ in
    $$
        S_0:=\{U\in H^1(\Omega_0):\Delta U=0\text{ in }\Omega_0\}
    $$
    with respect to the $L^2(\Omega_0)$ norm. Without loss of generality, we assume that $\Sigma\cap\Sigma_0\neq\varnothing$. Consequently, the functions in the class $S^\Dir$ and $S^\Neu$ can be uniquely determined by their boundary data in the class
    \begin{equation}\label{eq:Exact_Class}
        \begin{aligned}
            S_\bdy^\Dir :=& \left\{ f\in C^\infty_c(\Sigma): \int_{\partial\Omega}f n\cdot\nabla W_j \di\sigma = 0, \ j=1,...m \right\}; \\
            S_\bdy^\Neu :=& \left\{ g\in C^\infty_{c,0}(\Sigma): \int_{\partial\Omega} g W_j \di\sigma = \int_{\Sigma_j} g \di \sigma = 0, \ j=1,...m \right\}.
        \end{aligned}
    \end{equation}
    The proof is based on a duality argument. Given arbitrary function $V \in L^2(\Omega_0)$, we will show that, if $V$ is $L^2(\Omega_0)$-orthogonal to all the harmonic functions in the class $S^\Dir$ or $S^\Neu$, then $V$ is $L^2(\Omega_0)$-orthogonal to the class $S_0$. The Dirichlet and Neumann cases will be treated simultaneously. 

    \hspace{1em}

    \noindent\textbf{Orthogonality of boundary data.} To begin with, let us consider a zero extension of $V$,
    \begin{equation*}
        \tilde{V} = \left\{\begin{aligned}
            &V, && \text{in }\Omega_0, \\
            &0, && \text{in }\Omega\backslash\Omega_0,
        \end{aligned}\right.
    \end{equation*}
    and the function $\hat{V} \in H^{1}(\Omega)$ determined by 
    \begin{equation}\label{eq-Runge:DefhatV}
        \left\{\begin{aligned}
            &\Delta \hat{V} = \tilde{V}, && \text{in }\Omega, \\
            &\hat{V}|_{\partial\Omega}=0, && \text{(Dirichlet)} \\
            &n\cdot\nabla\hat{V}|_{\partial\Omega}=0. && \text{(Neumann)} \\
        \end{aligned}\right.
    \end{equation}
    Let $U$ be an arbitrary harmonic function in the class $S^\Dir$ or $S^\Neu$. By Green's formula, we have
    \begin{align*}
        0=\int_{\Omega_0} VU \di x = \int_{\Omega} \tilde{V}U \di x =& \int_{\Omega} \Delta \hat{V} U \di x \\
        =& \int_{\partial\Omega} n\cdot\nabla\hat{V} U \di \sigma - \int_{\Omega} \nabla \hat{V} \cdot \nabla U \di x \\
        =& \int_{\partial\Omega} n\cdot\nabla\hat{V} U \di \sigma - \int_{\partial\Omega} \hat{V} n\cdot\nabla U \di \sigma.
    \end{align*}
    The $L^2(\Omega_0)$-orthogonality condition on $V$ can thus be reduced to the following conditions:
    \begin{equation}\label{eq-Runge:Ortho}
        \begin{aligned}
            & \text{(Dirichlet)} && \int_{\partial\Omega} n\cdot\nabla\hat{V} f \di \sigma =0, \quad \forall f\in S_\bdy^\Dir; \\
            & \text{(Neumann)} && \int_{\partial\Omega} \hat{V} g \di \sigma =0, \quad \forall g\in S_\bdy^\Neu.
        \end{aligned}
    \end{equation}

    \hspace{1em}

    \noindent\textbf{Boundary value of $\hat{V}$.} In Dirichlet case, the orthogonality~\eqref{eq-Runge:Ortho} implies that 
    $$
        \exists (\lambda_1,...\lambda_m)\in\xR^m, \quad n\cdot\nabla\hat{V} = \sum_{j=1}^m \lambda_j n\cdot\nabla W_j,\quad \text{on }\Sigma.
    $$
    Recalling from the construction of $W_j$'s (see~\eqref{eq:HarMes}) and $\hat{V}$ (see~\eqref{eq-Runge:DefhatV}), the function $\hat{V} - \sum_{j=1}^m \lambda_j W_j$ vanishes on the component $\Sigma_0$ and is harmonic in the connected domain $\Omega\backslash\overline{\Omega}_0$. We can apply again the unique continuation property (see Proposition~\ref{prop:UniContin}) and deduce that
    $$
        \hat{V} = \sum_{j=1}^m \lambda_j W_j \quad \text{and} \quad n\cdot\nabla\hat{V} = n\cdot\nabla \left( \sum_{j=1}^m \lambda_j W_j \right) \quad \text{on }  \partial\Omega_0.
    $$
    In Neumann case, the orthogonality~\eqref{eq-Runge:Ortho} implies directly that
    $$
        \exists (\lambda_0,...\lambda_m)\in\xR^{m+1}, \quad \hat{V} = \sum_{j=0}^m \lambda_j W_j,\quad \text{on }\Sigma.
    $$
    Thus, $\hat{V}$ is constant on $\Sigma\cap\Sigma_0$, so is its normal derivative (see~\eqref{eq-Runge:DefhatV}). Again, by unique continuation property (Proposition~\ref{prop:UniContin}), we can deduce that $\hat{V}$ is constant in $\Omega\backslash\overline{\Omega}_0$.

    To summarize, in both cases, there exists a function $\hat{V}_0$ harmonic in the domain $\Omega$ such that
    $$
        \hat{V} = \hat{V}_0 \quad \text{and} \quad n\cdot\nabla\hat{V} = n\cdot\nabla \hat{V}_0 \quad \text{on }  \partial\Omega_0.
    $$

    \hspace{1em}

    \noindent\textbf{Proof of the density.}  Given any function $U$ in the class $S_0$, we compute
    \begin{align*}
        \int_{\Omega_0} VU \di x = \int_{\Omega_0} \Delta \hat{V}U \di x =& \int_{\partial\Omega_0} n\cdot\nabla\hat{V} U \di \sigma - \int_{\Omega_0} \nabla \hat{V} \cdot \nabla U \di x \\
        =& \int_{\partial\Omega_0} n\cdot\nabla\hat{V} U \di \sigma - \int_{\partial\Omega_0} \hat{V} n\cdot\nabla U \di \sigma \\
        =& \int_{\partial\Omega_0} n\cdot\nabla \hat{V}_0 U \di \sigma - \int_{\partial\Omega_0} \hat{V}_0 n\cdot\nabla U \di \sigma.
    \end{align*}
    Here we used the fact that $U$ is harmonic. Then, by reversing the calculation, we obtain
    \begin{equation*}
        \int_{\Omega_0} VU \di x = \int_{\Omega_0} \Delta \hat{V}_0 U \di x = 0,
    \end{equation*}
    since $\hat{V}_0$ is harmonic. This completes the proof.
\end{proof}
To end this part, we prove an endpoint-evaluation property of the classes $S^\Dir$ and $S^\Neu$.
\begin{proposition}\label{prop:RungeClEval}
    Let $\Omega$ and $\Sigma$ satisfy the assumptions as in Theorem~\ref{thm:Runge}, let $N\in\xN_+$, and let the $2N$ points $X_1^\initial,X_1^\final,...,X_N^\initial,X_N^\final\in\Omega$ be pairwise distinct. For $S=S^\Dir$ or $S=S^\Neu$, the linear mapping
    $$
        S \to \xR^N, \quad U \mapsto \Big( U(X_j^\final) - U(X_j^\initial) \Big)_{j=1}^N,
    $$
    is surjective.
\end{proposition}
\begin{proof}
    Choose mutually disjoint open balls around the $2N$ points, with closures contained in $\Omega$ and with the complement of their closed union connected in $\Omega$. We denote the union of these balls by $\Omega_0$ and define in $\Omega_0$ a locally harmonic function $\tilde{U}_j$ which equals $1$ in the ball around $X_j^\final$ and zero in all the other balls. By Theorem~\ref{thm:Runge}, applied in $\Omega_0$, we can approximate $\tilde{U}_j$ by a function $U_j$ in the class $S^\Dir$ or $S^\Neu$ such that
    $$
        U_j(X_j^\final) \in [1-\epsilon,1+\epsilon],
    $$
    and $U_j\in [-\epsilon,\epsilon]$ at the other points, for some $0<\epsilon\ll 1$. Thus, for each canonical basis vector $e_j$ of $\xR^N$, the image contains a vector $\tilde e_j$ arbitrarily close to $e_j$. For $\epsilon$ small enough these $N$ vectors are linearly independent. The linearity of the mapping then implies its surjectivity.
\end{proof}

\section{Stationary boundary control}\label{sect:PassCtrl}

We now combine the two ingredients of the preceding section to construct a stationary control. Unlike an active control, a stationary potential cannot compensate the point-vortex drift independently at each time. We therefore first construct a reference gradient flow which follows the prescribed arcs in a rescaled unit time. A localized Runge approximation produces admissible global fields close to this reference flow, and finitely many fixed harmonic corrections are then used to remove the resulting endpoint error. Finally, a time rescaling converts the small perturbation parameter into an arbitrarily short physical transfer time. 

\begin{theorem}\label{thm:PassCtrl}
    Let $\Omega\subset\xR^2$ be a domain satisfying hypothesis~\eqref{hyp:Domain} and $\Sigma\subset\partial\Omega$ be a nonempty relatively open subset. Consider arbitrary
    \begin{itemize}
        \item initial and final configurations $\bm X^\initial, \bm X^\final \in \mathcal{C}_N(\Omega)$ with $N\in\xN_+$ and $X_j^\initial\neq X_j^\final$ for all $j$;
        \item regular embedded real-analytic arcs $\Gamma_j$ parametrized by arc-length parameter,
        $$
            \Gamma_j: [0,L_j] \to \Omega, \quad s \mapsto \Gamma_j(s),
        $$
        where $L_j>0$ is its length.
    \end{itemize}
    If the prescribed curves $\Gamma_j$'s satisfy
    $$
        \begin{aligned}
            &\Gamma_j(0)=X_j^\initial,\qquad \Gamma_j(L_j)=X_j^\final,\qquad 1\leq j\leq N,\\
            &\Gamma_j([0,L_j])\cap\Gamma_k([0,L_k])=\varnothing,\qquad j\neq k,
        \end{aligned}
    $$
    then there exists $\varepsilon_0>0$ such that, for every $0<\varepsilon<\varepsilon_0$, one can find a stationary control $f^\varepsilon\in C^\infty_c(\Sigma)$ and a solution $\bm X^\varepsilon$ of~\eqref{eq:KR}-\eqref{eq:CtrlD} on $[0,\varepsilon]$ satisfying
    $$
        \bm X^\varepsilon(0)=\bm X^\initial,\qquad
        \bm X^\varepsilon(\varepsilon)=\bm X^\final.
    $$
    Moreover, the resulting potential of the background flow admits a harmonic conjugate. The same conclusion holds for a stationary Neumann control $g^\varepsilon\in C^\infty_c(\Sigma)$ satisfying~\eqref{eq:CtrlN_Comp}. 
    
    In either case, define the corresponding re-parametrized trajectory by
    $$
        \Gamma_j^\varepsilon(s):= X_j^\varepsilon(\varepsilon s/L_j),
        \qquad 0\leq s\leq L_j.
    $$
    The controls can be chosen so that
    $$
        \|\Gamma_j-\Gamma_j^\varepsilon\|_{C^k([0,L_j])}\to0
        \quad\text{as }\varepsilon\to0,\qquad k\in\xN,\quad 1\leq j\leq N.
    $$
\end{theorem}

\begin{remark}
    In particular, the desired transfer can be realized in any sufficiently small prescribed time $T>0$, by taking $\varepsilon=T$ in the theorem.
\end{remark}

\begin{remark}\label{rmk:ExclPermu}
    The disjointness condition on the $\Gamma_j$'s excludes a permutation of elements in configurations, such as
    $$
        N=2, \quad X_1^\initial = X_2^\final, \quad X_1^\final = X_2^\initial.
    $$
    The treatment of such a case remains open. The argument below cannot be applied to such an exchange, since the stationary harmonic field is constructed separately near each prescribed arc.
\end{remark}

\noindent\textbf{Strategy of the proof.} Recall that the evolution equation~\eqref{eq:KR} takes the form
\begin{equation}\label{eq:KR_gen}
    \dot{X}_j = V_j(\bm X) + \nabla \Phi(X_j),
\end{equation}
where
\begin{equation}\label{eq:DefVj}
    V_j(\bm X) = \sum_{k \neq j} \kappa_k \nabla_X^\perp G_\Omega(X_j, X_k) + \frac{\kappa_j}{2} \nabla^\perp R_\Omega(X_j) + \sum_{r=1}^m c_r(\bm X)\nabla^\perp W_r(X_j).
\end{equation}
The mappings $V_j$ are real-analytic on $\mathcal C_N(\Omega)$. Indeed, $G_\Omega$ is real-analytic away from the diagonal, while $R_\Omega$ and the harmonic measures are real-analytic in the interior of $\Omega$. The definition of $c_r(\bm X)$ then gives the asserted analyticity. No regularity of the boundary beyond the Lipschitz assumption is needed at points in the interior.
\\
We first construct a harmonic potential $\Phi^\refer$ near the prescribed arcs whose gradient flow generates those arcs simultaneously,
$$
    \dot{X}_j^\refer = \nabla \Phi^\refer(X_j^\refer), \quad X_j^\refer(0) = X^\initial_j, \quad X_j^\refer(1) = X_j^\final, \quad \text{with} \quad X_j^\refer([0,1]) = \Gamma_j([0,L_j]).
$$
Then the localized Runge theorem gives admissible global harmonic potentials which approximate this local one,
$$
    \dot{\bar X}_j^\app = \nabla \Phi^\varepsilon(\bar X_j^\app), \quad \bar X_j^\app(0) = X^\initial_j, \quad \bar X_j^\app(1) \text{ close to } X_j^\final.
$$
Since Runge approximation alone does not provide regular dependence on the approximation parameter $\varepsilon$, we interpolate the approximating sequence with a smooth cutoff obtained from a one-dimensional mollifier. Finally, we add $2N$ fixed harmonic corrections, rescale the time, and study the endpoint map,
$$
    (\varepsilon,\bm \alpha) \mapsto \bar{\bm X}^\varepsilon(1), \quad \text{where} \quad \dot{\bar X}^\varepsilon_j = \varepsilon V_j(\bar{\bm X}^\varepsilon) + \nabla \left( \Phi^\varepsilon + \sum_{l=1}^{2N}\alpha_l H_l \right) (\bar X_j^\varepsilon), \quad \bar X_j^\varepsilon(0) = X^\initial_j.
$$
An application of implicit function theorem corrects both the Runge error and the drift $V_j$, namely $\bar{\bm X}^\varepsilon(1) = \bm X^\final$. The desired solution to~\eqref{eq:KR_gen}-\eqref{eq:DefVj} can then be recovered from time-rescaling,
$$
    \bm X^\varepsilon(t) := \bar{\bm X}^\varepsilon(\varepsilon t), \quad 0 \le t \le \varepsilon.
$$
Here and below, the bar represents time-rescaled objects.

\subsection{Construction of the approximate dynamics}

Let us temporarily neglect the uncontrolled drift term $V_j$'s and construct an approximate control.

\medskip
\noindent\textbf{Local reference dynamics.} The first step is to construct a locally defined harmonic function whose gradient flow generated the prescribed dynamic. Set
$$
    X^\refer_j(t):=\Gamma_j(L_jt), \qquad 0\leq t\leq1,
$$
and $\bm X^\refer:=(X^\refer_1,...,X^\refer_N)$. Since $X^\refer_j$ is embedded, the mapping
$$
    [0,1] \to \xR^2, \quad t \mapsto \dot{X}^\refer_j(t)
    =L_j\dot\Gamma_j(L_jt)
$$
defines a real-analytic vector field along the regular embedded arc $\Gamma_j$. The local solvability of Laplace equation (cf. Lemma~\ref{lem:CurveData_HarFct}) gives a harmonic function $\Phi_j^\refer$, defined in a neighborhood of $\Gamma_j([0,L_j])$, such that
$$
    \nabla\Phi_j^\refer(X^\refer_j(t))=\dot{X}^\refer_j(t), \qquad \forall t\in[0,1].
$$
For every $j$, choose an open neighborhood $\Omega_j$ of $\Gamma_j([0,L_j])$ with closure contained in the domain of definition of $\Phi_j^\refer$. Since the prescribed curves $\Gamma_j([0,L_j])$'s are mutually disjoint, by choosing $\Omega_j$ close enough to $\Gamma_j([0,L_j])$, we may assume that the closures of $\Omega_j$ are mutually disjoint and that 
$$
    \Omega_0:=\bigsqcup_{j=1}^N\Omega_j, \qquad \Omega\backslash\overline{\Omega}_0\text{ is connected}.
$$
Define the harmonic function $\Phi^\refer$ on the disconnected set $\Omega_0$ by $\Phi^\refer|_{\Omega_j}=\Phi_j^\refer$. It follows that $\bm X^\refer = (X^\refer_1,...,X^\refer_N)$ verifies 
$$
    \dot{X}^\refer_j=\nabla\Phi^\refer(X^\refer_j), \quad X^\refer_j(0)=X_j^\initial, \quad X^\refer_j(1)=X_j^\final, \quad j=1,...,N.
$$
Moreover, the trajectory from $\bm X^\initial$ to $\bm X^\final$ coincides with the prescribed ones,
$$
    X^\refer_j([0,1]) = \Gamma_j([0,L_j]), \qquad j=1,...,N.
$$

\medskip
\noindent\textbf{Runge approximation.} The admissible classes $S^\Dir$ and $S^\Neu$ were defined in~\eqref{eq:RungeClDir}-\eqref{eq:RungeClNeu}. We use the notation $S=S^\Dir$ in the Dirichlet case and $S=S^\Neu$ in the Neumann case. Thus every member of $S$ is an admissible harmonic velocity potential, with boundary datum supported in $\Sigma$, and every member of $S^\Neu$ satisfies the compatibility condition~\eqref{eq:CtrlN_Comp}.
Fix compact tubular neighborhoods $K_j$ of the arcs such that
$$
    \Gamma_j([0,L_j])\subset\operatorname{int}(K_j), \qquad K_j\Subset\Omega_j,
$$
and set $K:=\bigsqcup_{j=1}^NK_j$. According to the refined Runge approximation theorem (see Theorem~\ref{thm:Runge}), there are functions $\Phi_n\in S$ such that
\begin{equation*}
    \|\Phi_n-\Phi^\refer\|_{C^n(K)}\leq2^{-n^2}, \qquad n\geq1.
\end{equation*}
A sequence of Runge approximations is not by itself sufficient for the later implicit-function argument, because it carries no differentiable dependence on its index. The super-algebraic accuracy imposed above allows us to extend the approximating sequence $\{\Phi_n\}$ to a family $\{\Phi^\varepsilon\}$ depending smoothly on some approximation parameter $\varepsilon$.
\\
More precisely, let $\rho\in C_c^\infty(]0,1[)$ be nonnegative with integral one and define
$$
    \theta(s):=\int_{-\infty}^s\rho(q)\di q, \qquad \varepsilon_n:=2^{-n}.
$$
The approximating family $\{\Phi^\varepsilon\}_{\varepsilon\in[-1,1]\backslash\{0\}}$ of $\Phi^\refer$ can be defined piecewise,
\begin{equation*}
    \begin{aligned}
        &\Phi^\varepsilon
        :=\theta\left(\frac{\varepsilon-\varepsilon_{n+1}}{\varepsilon_n-\varepsilon_{n+1}}\right)\Phi_n
        +\left[1-\theta\left(\frac{\varepsilon-\varepsilon_{n+1}}{\varepsilon_n-\varepsilon_{n+1}}\right)\right]\Phi_{n+1}, &&\text{if $\varepsilon\in[\varepsilon_{n+1},\varepsilon_n]$}; \\
        &\Phi^\varepsilon := \Phi^\refer, &&\text{if $\varepsilon=0$}; \\
        &\Phi^\varepsilon := \Phi^{-\varepsilon}, &&\text{if $\varepsilon<0$}.
    \end{aligned}
\end{equation*}
The cutoff $\theta$ is flat at $0$ and $1$, while the approximation errors $2^{-n^2}$ decrease faster than every power of $\varepsilon_n$. Consequently, for all $k\in\xN$,
\begin{equation*}
    [-1,1] \to C^k(K), \quad \varepsilon\mapsto\Phi^\varepsilon
    \quad\text{is a smooth mapping,}
\end{equation*}
and, for all $r,k\in\xN$,
\begin{equation*}
    \|\partial_\varepsilon^r(\Phi^\varepsilon-\Phi^\refer)\|_{C^k(K)} \to 0, \quad \text{as }\varepsilon\to 0.
\end{equation*}
Indeed, for fixed $r,k\in\xN$ with $n \ge k$, differentiation on the interval $[\varepsilon_{n+1},\varepsilon_n]$ gives
$$
   \sup_{\varepsilon\in [\varepsilon_{n+1},\varepsilon_n]} \|\partial_\varepsilon^r(\Phi^\varepsilon-\Phi^{\refer})\|_{C^k(K)} \leq C_{r,k}2^{r(n+1)-n^2}.
$$
In particular, $(\varepsilon,X)\mapsto\nabla\Phi^\varepsilon(X)$ is jointly smooth on a neighborhood of the arcs. Notice that for every $\varepsilon\neq0$, one has $\Phi^\varepsilon\in S$, since $S$ is a linear space.

\subsection{Endpoint correction and proof of the theorem}

The remaining issue is exactness at the final time. We introduce $2N$ fixed harmonic corrections, one for each endpoint coordinate, and choose their local gradient data so that the corresponding linearized variations are the canonical coordinate variations. After Runge approximation, the derivative of the endpoint map is arbitrarily close to the identity and is therefore invertible. The implicit function theorem then selects coefficients which correct simultaneously the Runge error and the uncontrolled point-vortex drift.

\medskip
\noindent\textbf{Endpoint map.} The final harmonic potential is chosen as a rescaling and correction of $\Phi^\varepsilon$ defined in the previous part,
$$
   \frac{1}{\varepsilon}\left(\Phi^\varepsilon+\sum_{l=1}^{2N}\alpha_lH_l\right) \in S,
$$
where coefficients $\alpha_l\in\xR$, harmonic functions $H_l\in S$, and the parameter $\varepsilon>0$ will be determined later. Through a rescaling in time $t \to t/\varepsilon$, the equation~\eqref{eq:KR_gen} can be reduced to the following equation for $\hat{\bm X} = (\hat X_1,...,\hat X_N)$,
\begin{equation}\label{eq:PerturEq}
    \left\{\begin{aligned}
        &\hat X_j = \varepsilon V_j(\hat{\bm X}) + \nabla \left( \Phi^\varepsilon + \sum_{l=1}^{2N} \alpha_l H_l \right)(\hat X_j), \\
        &\hat X_j(0) = X_j^\initial.
    \end{aligned}\right.
\end{equation}
Our goal is then to solve the endpoint equation
$$
    \mathscr T(\varepsilon,\bm\alpha):=\hat{\bm X}(1;\varepsilon,\bm \alpha)-\bm X^\final = 0, \quad \text{where $\hat{\bm X}$ solves~\eqref{eq:PerturEq} with parameter $(\varepsilon,\bm \alpha)$.}
$$
Let $\mathcal U\Subset\mathcal C_N(\Omega)$ be a neighborhood of the reference curve $\bm X^\refer([0,1])$ such that the $j$-th component of every configuration in $\overline{\mathcal U}$ belongs to $\operatorname{int}(K_j)$. The right-hand side of~\eqref{eq:PerturEq} is smooth jointly in $(\varepsilon,\bm\alpha,\hat{\bm X})$ on a neighborhood of $\{0\}\times\{\bm 0\}\times\overline{\mathcal U}$ and is therefore uniformly bounded and uniformly locally Lipschitz in $\hat{\bm X}$ there. Standard existence and smooth parameter dependence for ordinary differential equations, followed by Gronwall's inequality, show that, when the parameters are sufficiently small, the solutions remain close to the reference dynamics $\bm X^\refer$ and have smooth dependence on the parameters $(\varepsilon,\bm \alpha)$.

Therefore, the endpoint map $\mathscr{T}$ is well-defined and smooth on a neighborhood of $(0,\bm 0)$. Notice that, when $(\varepsilon,\bm \alpha)=(0,\bm 0)$, the solution to~\eqref{eq:PerturEq} is exactly the reference dynamics $\bm X^\refer$. In particular, 
$$
    \mathscr T(0,\bm 0)=\bm 0.
$$
It remains to show that
\begin{equation}\label{claim:EndPtMap_Invert}
    \exists H_l\in S, \text{ such that }\partial_{\bm \alpha}\mathscr{T}(0,\bm 0) \text{ is invertible}.
\end{equation}
Once it is proved, by implicit function theorem, we obtain $\varepsilon_0>0$ and a smooth implicit function $\bm \alpha(\varepsilon)$ such that
$$
    \mathscr T(\varepsilon,\bm\alpha(\varepsilon)) = 0, \quad \forall\ 0< \varepsilon <\varepsilon_0.
$$
The smooth dependence of the solution in parameters implies the convergence of trajectory\footnote{This convergence can also be proved by regarding them as the level set of the Hamiltonian function, in terms of the streamline function. Since the reference potential is non-degenerate, we can apply Lemma~\ref{lem:EqtoCurve} to deduce the convergence of trajectory.}:
$$
    \bar{\bm X}^\varepsilon:=\hat{\bm X}(\cdot;\varepsilon,\bm\alpha(\varepsilon)); \qquad \bar{\bm X}^\varepsilon\to \bm X^\refer\quad\text{in }C^k([0,1])\quad\text{for every }k\in\xN.
$$
Note that the corresponding solution to the original system~\eqref{eq:KR} (or equivalently~\eqref{eq:KR_gen}) is 
$$
    \bm X^\varepsilon(t):=\bar{\bm X}^\varepsilon(t/\varepsilon), \quad 0\leq t\leq\varepsilon; \qquad \bm X^\varepsilon(0) = \bm X^\initial, \quad \bm X^\varepsilon(\varepsilon) = \bm X^\final,
$$
where the control time $\varepsilon$ can be chosen arbitrarily small. 

\medskip
\noindent\textbf{Harmonic Corrections.} In this part, we precisely construct the harmonic corrections and prove the claim~\eqref{claim:EndPtMap_Invert}. From the smooth dependence of solution to~\eqref{eq:PerturEq} in parameters $(\varepsilon,\bm \alpha)$, we can compute directly the differentiation of the endpoint map $\mathscr{T}$ in $\bm \alpha$. Indeed, the differentiation takes the form
$$
    \partial_{\bm \alpha}\mathscr{T}(0,\bm 0) = \big(\bm Y_1(1)\mid\cdots\mid\bm Y_{2N}(1)\big),
$$
where
$$
    \bm Y_l(t):=\partial_{\alpha_l}\hat{\bm X}(t;0,\bm 0), \quad l=1,...,2N,
$$
is determined by the equation
\begin{equation}\label{eq-PassCtrl:DeriEndMap}
    \dot{\bm Y}_l=A(t)\bm Y_l+\big(\nabla H_l(X^\refer_1(t)),...,\nabla H_l(X^\refer_N(t))\big), \qquad \bm Y_l(0)=0.
\end{equation}
Here the $2N\times 2N$ matrix $A(t)$, defined by
$$
    A(t):=D_{\bm X}\mathcal F(\bm X^\refer(t)), \qquad \mathcal F(\bm X):=\big(\nabla\Phi^\refer(X_1),...,\nabla\Phi^\refer(X_N)\big),
$$
depends only on the reference data $\Phi^\refer$, $\bm X^\refer$ and is thus real-analytic in time.
\\
Applying again the local solvability of Laplace equation (see Lemma~\ref{lem:CurveData_HarFct}), one can construct harmonic functions $H_l^\loc$ locally defined near each reference dynamics $\Gamma_j([0,L_j])$'s such that the solution to~\eqref{eq-PassCtrl:DeriEndMap} is
$$
    \bm Y_l^\loc(t) = t \bm e_l, \qquad 0 \le t \le 1,
$$
where $\bm e_l$ denotes the $l$-th canonical vector of $\xR^{2N}$. Then Theorem~\ref{thm:Runge} gives harmonic functions $H_l\in S$ such that
$$
    \|H_l-H_l^{\mathrm{loc}}\|_{C^1(K)}<\eta,
$$
for arbitrarily small $\eta>0$. Here one may shrink the neighborhood $K$ of prescribed curves if necessary. In particular,
\begin{equation*}
    \sup_{ \substack{1\leq j\leq N \\ 0\leq t\leq1} }
    \left| \big(\nabla H_l(X^\refer_1(t)),...,\nabla H_l(X_N^\refer(t))\big) - \left( \dot{\bm Y}^\loc_l(t) - A(t) \bm Y^\loc_l(t) \right) \right|<\eta.
\end{equation*}
Consequently, by Gronwall's inequality, the resulting derivative $\bm Y_l$ is close to the one by local harmonic function
$$
    \sup_{0\leq t\leq1}|\bm Y_l(t)-t\bm e_l|\le C \eta, \qquad l=1,..,2N,
$$
for some constant $C>0$ independent of $\eta$. That is to say the differentiation $\partial_{\bm \alpha}\mathscr{T}(0,\bm 0)$ is close to identity,
$$
    \|\partial_{\bm\alpha}\mathscr T(0,\bm 0)-I_{2N}\|\leq C\eta, \qquad \forall\ 0< \eta \ll 1,
$$
and thus invertible. This proves the claim~\eqref{claim:EndPtMap_Invert} and thus Theorem~\ref{thm:PassCtrl}.

\section{Optimal control and geometric optimization}
\label{sect:OptCtrl}

In this section, we consider the optimal control problem for the total controlled length, namely the total length passed by controlled point vortices. When the control norm and the transfer time are both penalized, interior estimates give the existence of an optimal control under a uniform separation constraint. When these penalties are removed and only trajectory length remains, the problem is reduced to the optimal controlled length. This length is identical to the intrinsic distance in $\Omega$. When the control function is regular enough, the infimum of the optimal controlled length can not be attained, even for convex domains, where the intrinsic distance coincides with the Euclidean distance.

\medskip

We now introduce the optimal control problem for the controlled $N$-point vortex system~\eqref{eq:KR}-\eqref{eq:CtrlD} (or~\eqref{eq:CtrlN}), where the control is stationary. The cost functional penalizes simultaneously the total length traveled by the vortices, the transfer time, and the Sobolev norm of the boundary control.

Fix $s>0$. Recall from \eqref{eq:ConfigSet} the configuration space $\mathcal C_N(\Omega)$. For a stationary Dirichlet boundary control $f\in H^s(\partial\Omega)$, we denote by $\Phi$ the corresponding harmonic potential. We also recall here the controlled system (cf. equation~\eqref{eq:KR_gen}),
\begin{equation*}
    \dot{X}_j(t) = V_j(\bm X(t)) + \nabla\Phi(X_j(t)), \quad j=1,...,N,
\end{equation*}
where $V_j$ (see~\eqref{eq:DefVj}) denotes the uncontrolled velocity of the $j$-th vortex, including its interaction with the other vortices, the boundary contribution, and the fixed circulation sector. Recall also the notation for the $N$-vortex position
$$
    \bm X(t) = (X_1(t),...,X_N(t)) \in \mathcal{C}_N(\Omega).
$$
In order to avoid both collisions and contact with the boundary, for $\delta\geq0$ we introduce the refined configuration space
\begin{equation}\label{eq:ConfigSet_Ref}
    \mathcal C_{N,\delta}(\Omega):=\left\{\bm X=(X_1,...,X_N)\in\mathcal C_N(\Omega)\left|\ \begin{aligned}
        \;&
        \dist(X_j,\partial\Omega)\geq\delta,\quad 1\leq j\leq N,\\
        &|X_i-X_j|\geq\delta,\quad 1\leq i<j\leq N
    \end{aligned}\right.\right\}.
\end{equation}
For $\delta=0$, this definition gives
$$
    \mathcal C_{N,0}(\Omega)=\mathcal C_N(\Omega),
$$
and, in particular, $\mathcal C_{1,0}(\Omega)=\Omega$. When $\delta>0$, the additional conditions impose uniform separation from the boundary and from collisions.
Given initial and final configurations $\bm X^\initial,\bm X^\final \in \mathcal{C}_{N,\delta}(\Omega)$ and a nonempty relatively open subset $\Sigma\subset\partial\Omega$, we define the admissible class
\begin{equation*}
    \mathcal A_{\Sigma,\delta}^s(\bm X^\initial,\bm X^\final):=\left\{(f,T)\left|\ \begin{aligned}
        \;&
        f\in H^{s}(\partial\Omega),\quad T>0,\quad \Supp{f}\subset\overline{\Sigma},\\
        &\text{\eqref{eq:KR_gen} has a solution $\bm X$ with
        $\bm X(0)=\bm X^\initial$ and $\bm X(T)=\bm X^\final$,}\\
        &\bm X(t)\in\mathcal{C}_{N,\delta}(\Omega)
        \quad\text{for every }0\leq t\leq T
    \end{aligned}\right.\right\}.
\end{equation*}
The Neumann admissible class, denoted by the same symbol when no confusion is possible, is obtained by replacing the first line in the definition above by
$$
    g\in H^{s-1}(\partial\Omega),\quad T>0,\quad
    \Supp{g}\subset\overline\Sigma,\quad \int_{\partial\Omega} g \di \sigma = 0.
$$
The last condition is the weak interpretation of~\eqref{eq:CtrlN_Comp}, which should be understood as the distributional couple between $g$ and the constant function $1$. The corresponding harmonic potential is normalized by
$$
    \int_\Omega\Phi\di x=0.
$$
For simplicity, we give the proof in the Dirichlet case and record the necessary modifications for the Neumann case at the end. For $\mu, \beta \ge 0$, we consider the functional
\begin{equation}
\label{eq-cost-N-vortices}
J_{\mu,\beta}(f,T)
=
\sum_{j=1}^N
\int_0^T|\dot X_j(t)|\,dt
+
\beta T
+
\mu\|f\|_{H^s(\partial\Omega)}^2.
\end{equation}
The three terms have distinct roles. The first measures the total geometric length of the vortex trajectories, the second penalizes long transfer times, and the third penalizes the size and oscillations of the stationary boundary control. The restriction $s>0$ supplies the compactness of the interior harmonic extension operator used below.

\subsection{Existence of optimal controls}

We prove that, when $\beta,\mu>0$, so that both the boundary control and the transfer time are penalized, the functional~\eqref{eq-cost-N-vortices} admits a minimizer. The proof is based on the compactness of the resolution of boundary problems (see Proposition~\ref{prop:RegHarExt}). The positive control penalty bounds a minimizing sequence in the boundary Sobolev space, while the positive time penalty bounds the transfer times from above. The separation condition confines all trajectories to a compact subset of the collision-free configuration space. On this set, compactness of the interior harmonic extension gives uniform convergence of the controlled vector fields. The endpoint condition and the uniform speed bound prevent the limiting transfer time from vanishing.
\begin{theorem}
\label{thm-existence-optimal-N}
Let $\Omega\subset\mathbb R^2$ be a bounded connected domain with Lipschitz boundary, and let $\Sigma\subset\partial\Omega$ be nonempty and relatively open. Fix $N\in\xN_+$. Consider parameters $s,\delta>0$ and configurations $\bm X^\initial, \bm X^\final\in \mathcal{C}_{N,\delta}(\Omega)$ such that
$$
    \bm X^\initial\neq\bm X^\final, \qquad
    \mathcal{A}_{\Sigma,\delta}^s(\bm X^\initial, \bm X^\final) \neq \varnothing.
$$
Then, for every $\mu,\beta > 0$, the functional~\eqref{eq-cost-N-vortices} admits a minimizer
$$
    (f_\ast,T_\ast) \in \mathcal{A}_{\Sigma,\delta}^s(\bm X^\initial, \bm X^\final).
$$
The analogous conclusion holds for the compatible Neumann admissible class, with the control term in the functional replaced by $\mu\|g\|_{H^{s-1}(\partial\Omega)}^2$.
\end{theorem}
\begin{remark}
    The nonemptiness assumption is substantive when $\delta>0$ and $\Omega$ satisfies hypothesis~\eqref{hyp:Domain}. Theorem~\ref{thm:PassCtrl} implies nonemptiness whenever the reference arcs can be chosen with mutual and boundary separation strictly larger than $\delta$, after taking the approximation sufficiently close. On a Lipschitz boundary, the controls constructed there are restrictions of ambient smooth functions and therefore belong to $H^1(\partial\Omega)$. For higher indices it requires additional boundary regularity, for instance a smooth boundary, or a separate verification that the constructed datum belongs to the required Sobolev space. The absence of a permutation alone does not guarantee the required separation for a prescribed $\delta$.
\end{remark}

\begin{proof}
Let $(f_n,T_n)$ be a minimizing sequence. Since the admissible class is nonempty, there exists $C>0$ such that
$$
    J_{\mu,\beta}(f_n,T_n)\leq C.
$$
Consequently,
\begin{equation}
\label{eq-bounds-fn-Tn}
    T_n\leq\tfrac{C}{\beta},
    \qquad
    \|f_n\|_{H^s(\partial\Omega)}\leq\sqrt{\tfrac{C}{\mu}}.
\end{equation}
After extraction of a subsequence,
$$
    T_n\to T_\ast\geq0,
    \qquad
    f_n\rightharpoonup f_\ast\quad\text{weakly in }H^s(\partial\Omega),
$$
which implies
$$
    \Supp{f_\ast}\subset\overline\Sigma.
$$
Let $\Phi_n$ and $\Phi_\ast$ be the harmonic extensions of $f_n$ and $f_\ast$, respectively, and set
$$
    K_\delta:=\{x\in\overline\Omega:\dist(x,\partial\Omega)\geq\delta\}\Subset\Omega.
$$
The compactness of the interior harmonic extension operator in Proposition~\ref{prop:RegHarExt}, together with the weak convergence of $f_n$, gives
\begin{equation}
\label{eq-gradient-harmonic-compact}
    \nabla^k\Phi_n\longrightarrow\nabla^k\Phi_\ast
    \quad\text{uniformly on }K_\delta,\quad \forall k\in\xN.
\end{equation}
The set $\mathcal C_{N,\delta}(\Omega)$ is a closed subset of the compact set $K_\delta^N$, and is therefore compact. Define on this set
$$
    \mathcal F_n(\bm X):=\big(V_j(\bm X)+\nabla\Phi_n(X_j)\big)_{j=1}^N,
$$
and define $\mathcal F_\ast$ analogously. Since this set stays away from every singularity of the point-vortex field,
$$
    \mathcal F_n\longrightarrow\mathcal F_\ast
    \quad\text{uniformly on }\mathcal C_{N,\delta}(\Omega),
    \qquad
    \sup_n\|\mathcal F_n\|_{L^\infty(\mathcal C_{N,\delta}(\Omega))}\leq M
$$
for some $M>0$. If $\bm X_n$ is the trajectory associated with $(f_n,T_n)$, then
$$
    |\bm X^\final-\bm X^\initial|
    \leq\int_0^{T_n}|\dot{\bm X}_n(t)|\di t
    \leq MT_n.
$$
Since $\bm X^\initial\neq\bm X^\final$, it follows that
$$
    T_\ast\geq\frac{|\bm X^\final-\bm X^\initial|}{M}>0.
$$
Rescale all trajectories onto the fixed interval $[0,1]$. This avoids comparing solutions defined on different time intervals. The rescaled equations are equi-Lipschitz, so their limits preserve the state constraint and the two endpoints. Uniform convergence of the vector fields then permits passage to the limit in the equation and in the length term, while weak lower semicontinuity treats the control cost. More precisely,
\begin{equation*}
    \begin{aligned}
        &\bm Y_n(\tau):=\bm X_n(T_n\tau),\qquad
        \dot{\bm Y}_n(\tau)=T_n\mathcal F_n(\bm Y_n(\tau)),\\
        &\bm Y_n(0)=\bm X^\initial,\qquad \bm Y_n(1)=\bm X^\final.
    \end{aligned}
\end{equation*}
The sequence $\bm Y_n$ is equi-Lipschitz. By the Arzel\`a-Ascoli theorem, after extraction,
$$
    \bm Y_n\longrightarrow\bm Y_\ast
    \qquad\text{uniformly on }[0,1].
$$
Since $\mathcal C_{N,\delta}(\Omega)$ is closed,
$$
    \bm Y_\ast(\tau)\in\mathcal C_{N,\delta}(\Omega),
    \qquad
    \bm Y_\ast(0)=\bm X^\initial,
    \qquad
    \bm Y_\ast(1)=\bm X^\final.
$$
Moreover,
$$
    \bm Y_n(\tau)
    =\bm X^\initial+\int_0^\tau T_n\mathcal F_n(\bm Y_n(t))\di t.
$$
The uniform convergence of $\mathcal F_n$ on $\mathcal C_{N,\delta}(\Omega)$, the uniform convergence of $\bm Y_n$, and $T_n\longrightarrow T_\ast$ imply that
$$
    \bm Y_\ast(\tau)
    =\bm X^\initial+\int_0^\tau T_\ast\mathcal F_\ast(\bm Y_\ast(t))\di t.
$$
Consequently, $\bm Y_\ast\in C^1([0,1])$ and
$$
    \dot{\bm Y}_\ast=T_\ast\mathcal F_\ast(\bm Y_\ast).
$$
Thus,
$\bm X_\ast(t):=\bm Y_\ast(t/T_\ast)$, $0\leq t\leq T_\ast$, is an admissible trajectory with $(f_\ast,T_\ast)\in\mathcal A_{\Sigma,\delta}^s(\bm X^\initial,\bm X^\final)$.
Moreover,
$$
    \dot{\bm Y}_n=T_n\mathcal F_n(\bm Y_n)\longrightarrow T_\ast\mathcal F_\ast(\bm Y_\ast)=\dot{\bm Y}_\ast
$$
uniformly on $[0,1]$. Therefore,
\begin{align*}
    \sum_{j=1}^N\int_0^{T_n}|\dot X_{j,n}(t)|\di t
    &=\sum_{j=1}^N\int_0^1|\dot Y_{j,n}(\tau)|\di\tau\\
    &\longrightarrow\sum_{j=1}^N\int_0^1|\dot Y_{j,\ast}(\tau)|\di\tau
    =\sum_{j=1}^N\int_0^{T_\ast}|\dot X_{j,\ast}(t)|\di t.
\end{align*}
Finally, weak lower semicontinuity gives
$$
    \|f_\ast\|_{H^s(\partial\Omega)}^2
    \leq\liminf_{n\to\infty}\|f_n\|_{H^s(\partial\Omega)}^2,
$$
while $T_n\to T_\ast$. Hence,
$$
    J_{\mu,\beta}(f_\ast,T_\ast)
    \leq\liminf_{n\to\infty}J_{\mu,\beta}(f_n,T_n),
$$
which gives a minimizer.
\end{proof}

\begin{remark}
The separation condition $\bm X(t)\in\mathcal C_{N,\delta}(\Omega)$ has two purposes. It prevents collisions and contact with the boundary,
where the point-vortex vector field becomes singular, and it places
all trajectories in a compact subset on which the harmonic extensions
of bounded controls enjoy uniform interior estimates.
\end{remark}

\begin{remark}
The minimizer need not be unique. Although the control cost
$f\mapsto\mu\|f\|_{H^s(\partial\Omega)}^2$ is strictly convex, the dependence of the trajectories on the control
is nonlinear, and the endpoint constraint is not a convex constraint.
Thus different stationary harmonic fields may connect the same two
configurations with the same optimal cost.
\end{remark}

\subsection{Optimal controlled length}

In this part, we consider the limiting case $\mu=\beta=0$, in which only the controlled length is penalized. Furthermore, we consider the one-point case $N=1$ with $\delta=0$, allowing the trajectory to approach the boundary. The problem is then reduced to the optimal controlled length
\begin{equation}\label{eq:OptCtrlLen_Def}
    d_{\Sigma,\Omega}(X^\initial,X^\final) := \inf \int_0^T |\dot{X}(t)| \di t,
\end{equation}
where the infimum is taken over all control-time pairs in the admissible class $\mathcal A_{\Sigma,0}^s(X^\initial,X^\final)$. Here, we use the unbolded letters to denote the single point vortex.
In the Neumann formulation, the same symbol denotes the infimum after replacing the Dirichlet datum $f\in H^s(\partial\Omega)$ by a compatible datum $g\in H^{s-1}(\partial\Omega)$.
The dependence of $d_{\Sigma,\Omega}$ on $s$ and on the choice of the Dirichlet or Neumann formulation is suppressed in the notation. Although no positive separation margin is prescribed when $\delta=0$, the image of every individual continuous trajectory on a finite time interval is compactly contained in $\Omega$.

Recall that, given any two distinct points $X,X'\in\Omega$, we can define the intrinsic Euclidean distance
\begin{equation}\label{eq-intrinsic-distance}
    d_\Omega(X,X') := \inf_{\Gamma\in\mathcal C(X,X')} \operatorname{Length}(\Gamma),
\end{equation}
where
$$
    \mathcal C(X,X') := \left\{ \Gamma\in C^1([0,1];\Omega): \Gamma(0)=X,\ \Gamma(1)=X' \right\}.
$$
Thus,
\[
    d_\Omega(X,X') \geq |X-X'|,
\]
with equality whenever the segment $[X,X']$ is contained in
$\Omega$. In particular, if $\Omega$ is convex, $d_\Omega$ coincides with the Euclidean distance. The following lemma states that, in general, the intrinsic Euclidean distance can be approximated by the length of some real-analytic curves joining the given points.

\begin{lemma}\label{lem:NearMinAnalyticArc}
    Let $\Omega\subset\xR^2$ be a domain and let $X^\initial,X^\final\in\Omega$ be distinct. For every $\eta>0$, there exist $L>0$ and a regular embedded real-analytic arc
    $$
        \Gamma:[0,L]\to\Omega,
    $$
    parametrized by arc length, such that
    $$
        \Gamma(0)=X^\initial,\qquad \Gamma(L)=X^\final,\qquad
        L<d_\Omega(X^\initial,X^\final)+\eta.
    $$
\end{lemma}
\begin{proof}
    Choose $\gamma\in\mathcal C(X^\initial,X^\final)$ such that
    $$
        \operatorname{Length}(\gamma)
        <d_\Omega(X^\initial,X^\final)+\tfrac{\eta}{4}.
    $$
    Since $\gamma([0,1])$ is compactly contained in $\Omega$, a sufficiently fine polygonal approximation remains in $\Omega$ and has length arbitrarily close to that of $\gamma$. After an arbitrarily small perturbation of its interior vertices, it may be assumed to be in general position. Subdivide it at its finitely many intersection points. The resulting finite planar graph contains a simple edge path from $X^\initial$ to $X^\final$, whose realization is an embedded polygonal arc and whose length does not exceed that of the original polygonal path. Rounding its finitely many corners in pairwise disjoint balls compactly contained in $\Omega$, we obtain a regular embedded curve
    $$
        \gamma_0\in C^\infty([0,1];\Omega)
    $$
    with the same endpoints and
    $$
        \operatorname{Length}(\gamma_0)
        <d_\Omega(X^\initial,X^\final)+\tfrac{\eta}{2}.
    $$
    By the Weierstrass approximation theorem, choose real polynomials $P_n,Q_n$ such that
    $$
        (P_n,Q_n)\longrightarrow\dot\gamma_0
        \quad\text{uniformly on }[0,1].
    $$
    Set
    $$
        q_n(t):=(P_n(t),Q_n(t))
        +\gamma_0(1)-\gamma_0(0)
        -\int_0^1(P_n(t'),Q_n(t'))\di t'
    $$
    and
    $$
        \gamma_n(t):=\gamma_0(0)+\int_0^tq_n(t')\di t'.
    $$
    Then
    $$
        \gamma_n(0)=X^\initial,\qquad \gamma_n(1)=X^\final,
        \qquad \gamma_n\longrightarrow\gamma_0
        \quad\text{in }C^1([0,1]).
    $$
    Since the set of $C^1$ embeddings of the compact interval is open in the $C^1$ topology, for all sufficiently large $n$ the curve $\gamma_n$ is regular, embedded and contained in $\Omega$. Moreover,
    $$
        \left|\operatorname{Length}(\gamma_n)-\operatorname{Length}(\gamma_0)\right|
        \leq\int_0^1|\dot\gamma_n(t)-\dot\gamma_0(t)|\di t
        \longrightarrow0.
    $$
    Choose such an $n$ for which
    $$
        \operatorname{Length}(\gamma_n)
        <d_\Omega(X^\initial,X^\final)+\eta.
    $$
    Since $\gamma_n$ is polynomial and regular, its arc-length re-parametrization is real-analytic. This gives the required \mbox{arc $\Gamma$.} This completes the proof.
\end{proof}
There are two natural questions here:
\begin{itemize}
    \item Is the optimal controlled length $d_{\Sigma,\Omega}$ equal to the intrinsic Euclidean distance $d_\Omega$?
    \item Is the infimum~\eqref{eq:OptCtrlLen_Def} attained?
\end{itemize}
The answer to the first question is positive. Indeed, Theorem~\ref{thm:PassCtrl} not only provides an exact control, but also constructs the controlled trajectory arbitrarily close to any prescribed regular embedded real-analytic arc. By combining Lemma~\ref{lem:NearMinAnalyticArc} above and Theorem~\ref{thm:PassCtrl}, we obtain the following result.

\begin{theorem}\label{thm:OptCtrlLen}
    Let $\Omega\subset\mathbb R^2$ be a domain satisfying hypothesis~\eqref{hyp:Domain}. Let $\Sigma\subset\partial\Omega$ be a nonempty relatively open subset. Fix $s>0$, and let $X^\initial,X^\final\in\Omega$ with $X^\initial\neq X^\final$. Then
    \begin{equation}\label{eq:OptCtrlLen}
        d_{\Sigma,\Omega}(X^\initial,X^\final) = d_\Omega(X^\initial,X^\final).
    \end{equation}
    More precisely, for every $\varepsilon>0$, there exist $f_\varepsilon\in C_{c}^\infty(\Sigma)$ ($g_\varepsilon\in C_{c,0}^\infty(\Sigma)$ with~\eqref{eq:CtrlN_Comp} resp.) and $T_\varepsilon>0$ such that the solution to~\eqref{eq:KR}-\eqref{eq:CtrlD} (\eqref{eq:CtrlN} resp.) satisfies
    $$
        X(0)=X^\initial, \quad X(T_\varepsilon)=X^\final, \quad \text{and} \quad \int_0^{T_\varepsilon}|\dot X(t)|\di t < d_\Omega(X^\initial,X^\final)+\varepsilon.
    $$
\end{theorem}

Thus the value of the unregularized one-vortex problem is determined solely by the intrinsic geometry of $\Omega$. The boundary patch, the vortex intensity, and the fixed circulation sector influence the controls realizing the approximating trajectories, but not the limiting value of their lengths.

\subsection{Non-attainment in convex domains}

Theorem \ref{thm:OptCtrlLen}  identifies the infimum in the definition of $d_{\Sigma,\Omega}$ but does not imply that it is attained. If the intrinsic distance is not attained for the chosen endpoints, then neither is the controlled infimum. This may happen when every minimizing sequence approaches an obstructing part of the boundary; non-convexity by itself is not sufficient to decide attainment. In a convex domain, equality between the length of a curve and the Euclidean distance between its endpoints forces the curve to traverse the connecting segment monotonically. A stationary one-vortex trajectory is a level curve of the Hamiltonian formed by the Robin function and a harmonic conjugate of the control potential. Recall the Hamiltonian formulation given in Remark~\ref{rmk:Hamilton} and notice that the circulation term does not appear since a convex domain must be simply connected. Hence, attainment would force a nontrivial segment to lie in such a level set.

In this part, we will show that, when the background potential belongs to $H^1(\Omega)$ (or with higher regularity), the controlled trajectory can never be a segment, even when the control is applied on the full boundary $\partial\Omega$. Let us begin with the simplest case where $\Omega= \xD$ is the unit disk.
\begin{example}\label{ex:ImpoTraj}
    Let $\Omega= \xD$, $X^\initial=-1/2$, and $X^\final = 1/2$ with intensity $\kappa=2$. Here we identify $\xR^2$ with the complex plane $\xC$ and regard the real interval as a subset of the real axis. Assume, to the contrary, that there exists a harmonic potential $\Phi\in H^1(\xD)$ realizing a controlled trajectory from $X^\initial=-1/2$, to $X^\final = 1/2$. Then, by denoting its harmonic conjugate as $\Psi\in H^1(\Omega)$, we have
    $$
        R_\xD(X) + \Psi(X) = C \quad \text{on the interval }[-1/2,1/2],
    $$
    for some constant $C\in\xR$, where
    $$
        R_\xD(X)=-\frac{1}{2\pi}\log(1-|X|^2).
    $$
    By real-analyticity, the same identity holds for $X\in]-1,1[$, and hence
    $$
        \Psi(X)=C+\frac{1}{2\pi}\log(1-X^2)
        =C-\frac{1}{2\pi}\sum_{j=1}^\infty\frac{X^{2j}}{j}.
    $$
    If $h$ denotes the boundary trace of $\Psi$, the Poisson formula gives, for $X$ on the real axis near zero,
    $$
        \Psi(X)=\frac{1}{2\pi}\int_{-\pi}^{\pi}h(\theta)\di\theta
        +\frac{1}{\pi}\sum_{n=1}^\infty X^n\int_{-\pi}^{\pi}h(\theta)\cos(n\theta)\di\theta.
    $$
    Consequently, after modifying the constant $C$, the even part of $h$ must be
    \begin{equation*}
        \frac{h(\theta)+h(-\theta)}{2}
        =C-\frac{1}{2\pi}\sum_{j=1}^\infty\frac{\cos(2j\theta)}{j}
        =C+\frac{1}{2\pi}\log(2|\sin\theta|).
    \end{equation*}
    Its Fourier coefficients show that this function belongs to $H^s(\partial\xD)$ if and only if $s<1/2$. Hence, no such $\Psi$ belongs to $H^1(\xD)$.
\end{example}

The same obstruction holds in every bounded convex domain with real-analytic boundary. The obstruction comes from the behavior at the endpoints of the maximal chord containing the segment. Real-analyticity propagates the level-set identity along the whole chord. The Robin function then has logarithmic growth at the two boundary endpoints, so the harmonic part would have to cancel this growth. Such behavior is incompatible with finite Dirichlet energy, whose point-evaluation estimate permits only square-root logarithmic growth along a nontangential approach.

\begin{theorem}\label{thm-regularity-obstruction-convex}
Let $\Omega\subset\xR^2$ be a bounded convex domain with real-analytic boundary. For every harmonic function $\Psi\in H^1(\Omega)$, no nontrivial segment contained in $\Omega$ is contained in a level set of
$$
    H(X) := \frac{\kappa}{2}R_{\Omega}(X) + \Psi(X).
$$
Consequently, let $X^\initial,X^\final\in\Omega$ be distinct and let $\kappa\neq0$. For every nonempty relatively open subset $\Sigma\subset\partial\Omega$ and every $s\geq\frac12$, no stationary Dirichlet control $f\in H^s(\partial\Omega)$ with $\Supp{f}\subset\overline\Sigma$, and no compatible stationary Neumann control $g\in H^{s-1}(\partial\Omega)$ with $\Supp{g}\subset\overline\Sigma$, can generate a trajectory from $X^\initial$ to $X^\final$ of length $|X^\initial-X^\final|$.
\end{theorem}

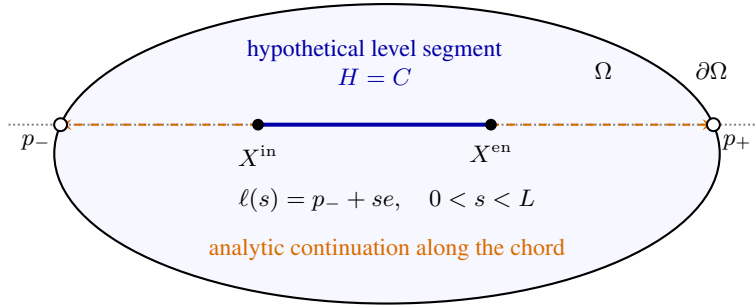
\begin{figure}[htbp]
\centering
\begin{tikzpicture}[scale=1.1, font=\small, >=stealth]
    \draw[thick, fill=blue!3] (0,0) ellipse (4 and 1.8);
    \node at (2.6,1.0) {$\Omega$};
    \node at (3.9,1.0) {$\partial\Omega$};
    \coordinate (pminus) at ({-4*sqrt(1-(0.35/1.8)^2)},0.35);
    \coordinate (pplus) at ({4*sqrt(1-(0.35/1.8)^2)},0.35);
    \coordinate (initial) at (-1.55,0.35);
    \coordinate (final) at (1.25,0.35);

    \draw[gray, densely dotted] (-4.55,0.35) -- (4.55,0.35);
    \draw[blue!65!black, line width=1.5pt] (initial) -- (final);
    \draw[->, orange!85!black, dashed, thick] (initial) -- (pminus);
    \draw[->, orange!85!black, dashed, thick] (final) -- (pplus);
    \fill (initial) circle (2pt) node[below=4pt] {$X^\initial$};
    \fill (final) circle (2pt) node[below=4pt] {$X^\final$};
    \draw[fill=white] (pminus) circle (2.2pt) node[below left] {$p_-$};
    \draw[fill=white] (pplus) circle (2.2pt) node[below right] {$p_+$};
    \node[blue!65!black, align=center] at (-0.15,1.1)
        {hypothetical level segment\\$H=C$};
    \node at (0,-0.55) {$\ell(s)=p_-+se,\quad 0<s<L$};
    \node[orange!85!black] at (0,-1.15)
        {analytic continuation along the chord};
\end{tikzpicture}
\caption{Hypothetical segment level set in
Theorem~\ref{thm-regularity-obstruction-convex}. If $H$
were constant on $[X^\initial,X^\final]$, real-analyticity would extend
this identity to the open maximal chord $(p_-,p_+)$.}
\label{fig:convex-level-segment}
\end{figure}

\begin{proof}
    Suppose, to the contrary, that $H$ is constant on a nontrivial segment $[X^\initial,X^\final]\subset\Omega$.
    Denote by $p_-,p_+\in\partial\Omega$ the endpoints of the maximal chord contained in the line through $X^\initial$ and $X^\final$, and set
    $$
        L:=|p_+-p_-|,\qquad e:=\frac{p_+-p_-}{L}.
    $$
    We parametrize the chord by $\ell(s)=p_-+se$, $0<s<L$. The real-analytic function $s\mapsto H(\ell(s))$ is constant on the subinterval corresponding to $[X^\initial,X^\final]$. By the identity principle, it is constant for all $0<s<L$. Consequently,
    \begin{equation}\label{eq-H-constant-maximal-chord}
        \Psi(\ell(s)) = C - \frac{\kappa}{2}R_\Omega(\ell(s))
    \end{equation}
    holds for some constant $C\in\xR$.
    We next recall the boundary behavior of the Robin function. Let $F:\Omega\to\mathbb D$ be a conformal map. By conformal invariance of the Dirichlet Green function,
    $$
        G_\Omega(z,\zeta) = G_{\mathbb D}(F(z),F(\zeta)).
    $$
    With the convention used throughout this paper,
    $$
        G_{\mathbb D}(w,\eta) = \frac{1}{2\pi} \log\left| \frac{w-\eta}{1-\overline{\eta}w} \right|,
    $$
    one obtains
    \[
        R_\Omega(z)
        =
        \frac{1}{2\pi}
        \log\left(
        \frac{|F'(z)|}{1-|F(z)|^2}
        \right).
    \]
    The real-analyticity of $\partial\Omega$ implies that $F$ extends conformally across the boundary. It follows that
    \[
        R_\Omega(z)
        =
        -\frac{1}{2\pi}
        \log \dist(z,\partial\Omega)+O(1),
        \qquad
        z\to\partial\Omega.
    \]
    Since $\Omega$ is convex, the line containing the maximal chord intersects $\partial\Omega$ transversally at $p_-$ and $p_+$. Indeed, if the chord direction were tangent to $\partial\Omega$ at either endpoint, the corresponding tangent line would be a supporting line of the convex set $\overline\Omega$. This is impossible because the same line contains the interior points of the chord. The nonzero angle of intersection and the $C^1$ regularity of the boundary therefore imply
    \[
        \dist(X,\partial\Omega)
        \asymp
        |X-p_\pm|
    \]
    as $X\to p_\pm$ along the chord. Therefore, the following asymptotic formulas hold along the chord,
    \[
        R_\Omega(X)
        =
        -\frac{1}{2\pi}\log|X-p_\pm|+O(1).
    \]
    Together with \eqref{eq-H-constant-maximal-chord}, this gives
    \begin{equation}
    \label{eq-log-growth-optimal}
        \Psi(X)
        =
        \frac{\kappa}{4\pi}\log|X-p_\pm|+O(1),
        \qquad
        X\to p_\pm,
    \end{equation}
    along the maximal chord.
    On the other hand, recall that if {$\Phi$} is harmonic in the unit disc $\mathbb D$ and
    \[
    \int_{\mathbb D}|\nabla \Phi|^2\,dx<\infty,
    \]
    then
    \begin{equation*}
        |\Phi(z)| \lesssim_\Phi 1+ \sqrt{\log\tfrac{1}{1-|z|}} , \qquad z\in\mathbb D.
    \end{equation*}
    This is the classical point-evaluation estimate for the Dirichlet
    space; see, for instance, \cite[Chapter~1]{ElFallah}. We also refer to
    \cite[Chapters~3 and~6]{ElFallah} for the boundary behavior and conformal invariance of the Dirichlet space. By conformal invariance of the Dirichlet integral, the same estimate can be transferred to any bounded simply connected domain with smooth boundary. More precisely, for every $p\in\partial\Omega$, $\Psi\in H^1(\Omega)$ implies that
    \begin{equation}\label{eq-dirichlet-pointwise-domain}
        |\Psi(X)| \lesssim_\Psi 1+ \sqrt{\log\tfrac{1}{|X-p|}}, \quad \text{as $X\to p$ nontangentially.}
    \end{equation}
    Indeed, when $F:\Omega\to\mathbb D$ is a conformal map and $\zeta=F(p)$, the Dirichlet integral is conformally invariant, while the smooth extension of $F$ to the boundary gives
    \[
        1-|F(X)| \asymp |X-p|
    \]
    along every nontangential approach to $p$. The maximal chord gives such an approach. However, this estimate~\eqref{eq-dirichlet-pointwise-domain} is incompatible with the logarithmic behavior~\eqref{eq-log-growth-optimal}. This completes the proof of the desired result.
\end{proof}

\appendix

\section{Technical results}\label{app:Tech}

In this appendix, we collect the technical results frequently used in this paper.

\begin{proposition}[Regularity of harmonic extension]\label{prop:RegHarExt}
    Let $\Omega\subset\xR^2$ be a bounded connected domain with Lipschitz boundary and $s>0$. The trace operator $\Tr:H^{s+1/2}(\Omega)\to H^{s}(\partial\Omega)$ is continuous and surjective. For every $f\in H^{s}(\partial\Omega)$, the variational Dirichlet problem
    $$
        \left\{\begin{aligned}
            & \Delta \Phi = 0, &&\text{in $\Omega$}, \\
            & \Tr \Phi = f, &&
        \end{aligned}\right.
    $$
    has a unique solution in $H^{s+1/2}(\Omega)$ and
    $$
        \|\Phi\|_{H^{s+1/2}(\Omega)}\lesssim\|f\|_{H^{s}(\partial\Omega)}.
    $$
\end{proposition}
\begin{remark}
    For the Neumann boundary condition, write
    $H^{-1/2}_\diamond(\partial\Omega):=\{g\in H^{-1/2}(\partial\Omega):\langle g,1\rangle=0\}$.
    The Dirichlet-to-Neumann operator in~\eqref{eq:DtN} induces a bounded isomorphism
    \begin{equation*}
        \mathcal N:H^{\sigma+1/2}(\partial\Omega)/\xR\longrightarrow H^{\sigma-1/2}_\diamond(\partial\Omega), \quad \forall \sigma\ge 0.
    \end{equation*}
    By duality, this can be extended to the case with $\sigma \ge -1/2$. Therefore the result of the proposition above can be stated for Neumann boundary condition.
\end{remark}
\begin{remark}
    We also emphasize the strong interior regularizing effect of the Laplace equation. Although the regularity of $\Phi$ up to $\partial\Omega$ is determined by that of the boundary datum $f$, harmonicity gives $\Phi\in C^\omega(\Omega)$. In addition, for every compact set $K\Subset\Omega$ and every $k\in\xN$, interior elliptic estimates yield
    $$
        \| \Phi \|_{C^k(K)} \lesssim_{K,k} \| f \|_{H^{s}(\partial\Omega)}.
    $$
    Moreover, the linear mapping
    $$
        H^s(\partial\Omega) \to C^k(K), \quad f \to \Phi
    $$
    is compact.
\end{remark}

\hspace{1em}

\begin{proposition}[Unique continuation]\label{prop:UniContin}
    Let $\Omega\subset\xR^2$ be a domain with Lipschitz boundary. Let $\Sigma\subset\partial\Omega$ be a nonempty relatively open subset of the boundary. If $U\in H^1(\Omega)$ is harmonic inside $\Omega$ such that
    \begin{equation}\label{eq-UniContin:ZeroCond}
        U = n\cdot\nabla U = 0, \quad \text{on }\Sigma,
    \end{equation}
    where $n$ denotes the normal direction and the second equality means that the weak conormal derivative pairs to zero with every function in $C_c^\infty(\Sigma)$, then $U$ is identically zero on $\Omega$.
\end{proposition}
\begin{proof}
    Let $x_0\in\Sigma$ and choose a ball $B$ in a Lipschitz boundary chart, centered at $x_0$, such that $B\cap\partial\Omega\subset\Sigma$. Define $\widetilde U=U$ on $B\cap\Omega$ and $\widetilde U=0$ on $B\backslash\overline\Omega$. The zero trace implies that $\widetilde U\in H^1(B)$. For every $\varphi\in C_c^\infty(B)$, Green's formula and the zero weak conormal derivative give
    $$
        \int_B\nabla\widetilde U\cdot\nabla\varphi\di x
        =\int_{B\cap\Omega}\nabla U\cdot\nabla\varphi\di x
        =\left\langle n\cdot\nabla U,\varphi|_{\partial\Omega}\right\rangle=0.
    $$
    Hence
    $$
        \Delta\widetilde U=0\quad\text{in }B
    $$
    in the sense of distributions. Since $\widetilde U$ vanishes in the nonempty open set $B\backslash\overline\Omega$, interior unique continuation gives $U=0$ in $B\cap\Omega$. Interior unique continuation in the connected domain $\Omega$ then yields $U\equiv0$.
\end{proof}
Stronger boundary unique-continuation results under positive-measure hypotheses are discussed in~\cite{vita2026unique,lin1991nodal,xavier2023unique}.

\hspace{1em}

\begin{lemma}[Approximation of level sets]\label{lem:EqtoCurve}
    Consider a domain $\Omega\subset\xR^2$ and a real-valued smooth function $H$ on $\Omega$. Let $\Gamma:[0,L]\to\Omega$ be a smooth regular embedded arc parametrized by arc length, and set $X^\initial:=\Gamma(0)$ and $X^\final:=\Gamma(L)$. Assume that
    $$
        \Gamma([0,L])\subset\{H=H(X^\initial)\}
    $$
    and that $\nabla H$ does not vanish on this arc.
    If there exist smooth functions $H^\epsilon:\Omega\to\xR$ with parameter $\epsilon>0$ such that
    \begin{itemize}
        \item for all $\epsilon>0$, $H^\epsilon(X^\initial)=H^\epsilon(X^\final)$;
        \item for all $k\in\xN$ and every compact subset $K\Subset \Omega$
        $$
            \| H - H^\epsilon \|_{C^k(K)} \to 0, \quad \text{as }\epsilon\to 0;
        $$
    \end{itemize}
    Then, for all $\epsilon>0$ small enough, the level set $\{ H^\epsilon = H^\epsilon(X^\initial) \}$ contains a curve $\Gamma^\epsilon:[0,L]\to\Omega$ from $X^\initial$ to $X^\final$ such that, for all $k\in\xN$,
    $$
        \| \Gamma - \Gamma^\epsilon \|_{C^k([0,L])} \to 0, \quad \text{as }\epsilon \to 0.
    $$
    The parameter of $\Gamma^\epsilon$ need not be its arc-length parameter.
\end{lemma}
\begin{proof}
    Let $n$ denote a smooth unit normal along $\Gamma$, and introduce tubular coordinates near the curve,
    \begin{equation}\label{eq:LocCoor}
        \Xi: (s,\iota)\in[0,L]\times [-\delta,\delta] \mapsto \Gamma(s) + \iota n(s) \in \Omega
    \end{equation}
    where $\delta>0$ is sufficiently small. Since $H$ is constant on $\Gamma$ and $\nabla H\neq0$ there, $\nabla H(\Gamma(s))$ is normal to the arc. After choosing the sign of $n$, compactness gives
    $$
        \partial_\iota(H\circ\Xi)(s,0)=\nabla H(\Gamma(s))\cdot n(s)\geq2c
    $$
    for some $c>0$ and every $s\in[0,L]$. Shrinking $\delta$ and using the $C^1$ convergence of $H^\epsilon$ to $H$, we obtain
    $$
        \partial_\iota(H^\epsilon\circ\Xi)(s,\iota)\geq c
    $$
    throughout the fixed tube, uniformly for all sufficiently small $\epsilon$. Moreover, $H^\epsilon(\Gamma(s))-H^\epsilon(X^\initial)$ converges uniformly to zero. For small enough $\epsilon$, its absolute value is less than $c\delta/2$. Integrating the preceding lower bound from $0$ to $\pm\delta$ shows that
    $$
        H^\epsilon\circ\Xi(s,-\delta)<H^\epsilon(X^\initial)
        <H^\epsilon\circ\Xi(s,\delta)
    $$
    for every $s\in[0,L]$. Hence the intermediate value theorem and strict monotonicity give a unique small function $h^\epsilon$ such that the level set is represented by $\iota=h^\epsilon(s)$, where $h^\epsilon(s)$ is the unique solution of
    $$
        H^\epsilon\circ\Xi(s,\cdot) = H^\epsilon(X^\initial) \quad \text{on }[-\delta,\delta]
    $$
    for each fixed $s\in[0,L]$. The implicit function theorem shows that $h^\epsilon$ is smooth. The endpoint identity in the assumptions and uniqueness of the small root give $h^\epsilon(0)=h^\epsilon(L)=0$.
    From the identity
    \begin{align*}
        \int_{0}^{h^\epsilon(s)}\partial_\iota(H^\epsilon\circ\Xi)(s,\iota')\di\iota'
        &=H^\epsilon\circ\Xi(s,h^\epsilon(s))-H^\epsilon\circ\Xi(s,0)\\
        &=H^\epsilon(X^\initial)-H^\epsilon(\Gamma(s)),
    \end{align*}
    we have
    \begin{align*}
        c|h^\epsilon(s)|
        &\lesssim |H^\epsilon(X^\initial)-H^\epsilon(\Gamma(s))|\\
        &\leq |H^\epsilon(X^\initial)-H(X^\initial)|
        +|H(\Gamma(s))-H^\epsilon(\Gamma(s))|
    \end{align*}
    uniformly in $\epsilon$, where the right-hand side tends to zero as $\epsilon\to0$. Differentiating the implicit identity and arguing inductively gives $h^\epsilon\to0$ in every $C^k([0,L])$. Therefore $\Gamma^\epsilon(s):=\Xi(s,h^\epsilon(s))$ has the required endpoints and convergence. It is regular and embedded for all sufficiently small $\epsilon$ by $C^1$ closeness to $\Gamma$.
\end{proof}

\printbibliography[heading=bibliography,title=References]

\end{document}